\documentclass[12pt]{amsart}
\usepackage{amssymb,amsmath,amsthm,bookmark,hyperref,geometry}
\title{The topological pigeonhole principle for ordinals}
\author{Jacob Hilton}
\address{School of Mathematics\\University of Leeds\\Leeds LS2 9JT\\UK}
\email{mmjhh@leeds.ac.uk}
\thanks{This research was conducted under the supervision of John K. Truss with the support of an EPSRC Doctoral Training Grant Studentship.}
\theoremstyle{plain}
\newtheorem{theorem}{Theorem}[section]
\newtheorem{lemma}[theorem]{Lemma}
\newtheorem{proposition}[theorem]{Proposition}
\newtheorem{corollary}[theorem]{Corollary}
\theoremstyle{definition}
\newtheorem{definition}[theorem]{Definition}
\newtheorem*{notation}{Notation}

\begin{document}

\begin{abstract}
Given a cardinal $\kappa$ and a sequence $\left(\alpha_i\right)_{i\in\kappa}$ of ordinals, we determine the least ordinal $\beta$ (when one exists) such that the topological partition relation
\[\beta\rightarrow\left(top\,\alpha_i\right)^1_{i\in\kappa}\]
holds, including an independence result for one class of cases. Here the prefix ``$top$'' means that the homogeneous set must have the correct topology rather than the correct order type. The answer is linked to the non-topological pigeonhole principle of Milner and Rado.
\end{abstract}

\maketitle

\section{Introduction}

Dirichlet's pigeonhole principle describes how if a large enough number of items are put into few enough containers, then there is some container that contains not too small a number of items. Milner and Rado \cite{milnerrado} generalised this principle to ordinals by answering the following question. Given a cardinal $\kappa$ and a sequence $\left(\alpha_i\right)_{i\in\kappa}$ of ordinals, what is the least ordinal $\beta$ such that whenever $\beta$ is partitioned into $\kappa$ pieces, there is some $i\in\kappa$ such that the order type of the $i\text{th}$ piece is at least $\alpha_i$? In the partition calculus of Erd\H{o}s and Rado, this is written as $\beta\rightarrow\left(\alpha_i\right)^1_{i\in\kappa}$.

With the development of structural Ramsey theory, questions of this character have become much more diverse. In particular, people have asked for various topological spaces $X$ and $Y$ whether it is the case that whenever $Y$ is partitioned into $\kappa$ pieces, one of the pieces has a subspace homeomorphic to $X$. This is written as $Y\rightarrow\left(top\,X\right)^1_\kappa$. Ordinals have been studied in this context by endowing them with the order topology, but the focus has been on a fairly small number of key cases.

In this article we answer the topological question for ordinals in full generality. Given a cardinal $\kappa$ and a sequence $\left(\alpha_i\right)_{i\in\kappa}$ of ordinals, we determine the least ordinal $\beta$ (when one exists) such that whenever $\beta$ is partitioned into $\kappa$ pieces, there is some $i\in\kappa$ such that the $i\text{th}$ piece has a subset homeomorphic to $\alpha_i$ under the subspace topology. We write this as $\beta\rightarrow\left(top\,\alpha_i\right)^1_{i\in\kappa}$.

Many of the cases of this question that have already been answered are given in a summary article by Weiss \cite{weiss}. These include the following; here we write $\beta\rightarrow\left(top\,\alpha\right)^1_\kappa$ for $\beta\rightarrow\left(top\,\alpha_i\right)^1_{i\in\kappa}$ when $\alpha_i=\alpha$ for all $i\in\kappa$.

\begin{enumerate}
\item\label{multiplicativelyindivisbleproperty}
If $\alpha\in\omega_1\setminus\left\{0,1\right\}$, then $\alpha\rightarrow\left(top\,\alpha\right)^1_2$ if and only if $\alpha=\omega^{\omega^\beta}$ for some $\beta\in\omega_1$.
\item
$\omega_1\rightarrow\left(top\,\alpha\right)^1_{\aleph_0}$ for all $\alpha\in\omega_1$.
\item
If $\alpha\in\omega_2$, then $\alpha\nrightarrow\left(top\,\omega_1\right)^1_2$.
\item
If $V=L$ then $\alpha\nrightarrow\left(top\,\omega_1\right)^1_2$ for all ordinals $\alpha$ \cite{prikrysolovay}, but it is equiconsistent with the existence of a Mahlo cardinal that $\omega_2\rightarrow\left(top\,\omega_1\right)^1_2$ \cite{shelah}.
\end{enumerate}

Note that for $\alpha\in\omega_1\setminus\left\{0\right\}$, $\alpha\rightarrow\left(\alpha\right)^1_2$ if and only if $\alpha=\omega^\beta$ for some $\beta\in\omega_1$. Thus in view of (\ref{multiplicativelyindivisbleproperty}) it is natural to ask whether or not there is a link between the topological and non-topological pigeonhole principles. Our main breakthrough is a full analysis of the case in which $\kappa$ is finite and $\alpha_i\in\omega_1$ for all $i\in\kappa$, where we bring this link to light. Here $\#$ denotes the natural sum operation, and $P^{top}\left(\alpha_i\right)_{i\in\kappa}$ (respectively $P^{ord}\left(\alpha_i\right)_{i\in\kappa}$) denotes the least ordinal $\beta$ such that $\beta\rightarrow\left(top\,\alpha_i\right)^1_{i\in\kappa}$ (respectively $\beta\rightarrow\left(\alpha_i\right)^1_{i\in\kappa}$).

\begin{theorem}\label{link}
Let $\alpha_1,\alpha_2,\dots,\alpha_k\in\omega_1\setminus\left\{0\right\}$.
\begin{enumerate}
\item\label{linkpartsuccessor}
$P^{top}\left(\omega^{\alpha_1}+1,\omega^{\alpha_2}+1,\dots,\omega^{\alpha_k}+1\right)=\omega^{\alpha_1\#\alpha_2\#\cdots\#\alpha_k}+1$.
\item\label{linkpartpowers}
$P^{top}\left(\omega^{\alpha_1},\omega^{\alpha_2},\dots,\omega^{\alpha_k}\right)=\omega^{P^{ord}\left(\alpha_1,\alpha_2,\dots,\alpha_k\right)}$.
\end{enumerate}
\end{theorem}

We prove this (in Theorems \ref{linksuccessor} and \ref{linkpowers}) using a result of Weiss \cite[Theorem 2.3]{baumgartner}. This was published by Baumgartner, who used it to show that $\omega^{\omega^\alpha\cdot\left(2m+1\right)}\rightarrow\left(top\,\omega^{\omega^\alpha\cdot\left(m+1\right)}\right)^1_2$ for all $m\in\omega$ and all $\alpha\in\omega_1$ \cite[Corollary 2.5]{baumgartner}. The above theorem greatly generalises this, thereby utilising the full potential of Weiss's result.

The only case in which we provide no answer in $\mathsf{ZFC}$ is when $1<\alpha_i\leq\omega_1$ for all $i\in\kappa$ and we have equality in at least two instances. In this case we have an independence result. Prikry and Solovay showed that if $V=L$ then $\alpha\nrightarrow\left(top\,\omega_1\right)^1_2$ for all ordinals $\alpha$, from which it follows that it is consistent for $P^{top}\left(\alpha_i\right)_{i\in\kappa}$ not to exist. On the other hand, we show that $P^{top}\left(\alpha_i\right)_{i\in\kappa}\geq\operatorname{max}\left\{\omega_2,\kappa^+\right\}$ and deduce from a result of Shelah that it is consistent to have equality in every case, assuming the consistency of the existence of a supercompact cardinal.

It remains open which intermediate values can consistently be taken by these topological pigeonhole numbers, and whether or not Shelah's consistency result can be strengthened to equiconsistency.

\section{Preliminaries}

\subsection{Partition relation notation}

\begin{notation}
We use the von Neumann definitions of ordinals and cardinals, namely that each is the set of all smaller ordinals, and unless stated otherwise all arithmetic will be ordinal arithmetic. We denote the cardinal successor of a cardinal $\kappa$ by $\kappa^+$, and we use interval notation in the usual fashion, so that for example if $\alpha$ and $\beta$ are ordinals then $\left[\alpha,\beta\right)=\left\{x:\alpha\leq x<\beta\right\}$. We denote the cardinality of a set $X$ by $\left|X\right|$, and we use the symbol $\cong$ to denote the homeomorphism relation.
\end{notation}

In this article we study two different notions of partition relation. We begin by defining both of these; here $n$ is a positive integer, $\kappa$ is a cardinal, and $\left[X\right]^n$ denotes the set of subsets of $X$ of size $n$ (except that for simplicity we take $\left[X\right]^1$ to be $X$).

\begin{definition}
Let $\beta$ be an ordinal and let $\alpha_i$ be an ordinal for each $i\in\kappa$. We write
\[\beta\rightarrow\left(\alpha_i\right)^n_{i\in\kappa}\]
to mean that for every function $c:\left[\beta\right]^n\rightarrow\kappa$ there exists some subset $X\subseteq\beta$ and some $i\in\kappa$ such that $X$ is order-isomorphic to $\alpha_i$ and $\left[X\right]^n\subseteq c^{-1}\left(\left\{i\right\}\right)$.
\end{definition}

\begin{definition}
Let $Y$ be a topological space and let $X_i$ be a topological space for each $i\in\kappa$. We write
\[Y\rightarrow\left(top\,X_i\right)^n_{i\in\kappa}\]
to mean that for every function $c:\left[Y\right]^n\rightarrow\kappa$ there exists some subspace $X\subseteq Y$ and some $i\in\kappa$ such that $X\cong X_i$ and $\left[X\right]^n\subseteq c^{-1}\left(\left\{i\right\}\right)$.
\end{definition}

When ordinals are used as topological spaces, they are assumed to have the order topology. When $\alpha_i=\alpha$ for all $i\in\kappa$ we write $\beta\rightarrow\left(\alpha\right)^n_\kappa$ for $\beta\rightarrow\left(\alpha_i\right)^n_{i\in\kappa}$, and similarly for the topological relation.

The function $c$ in these definitions will often be referred to as a \emph{colouring}, and we may say that \emph{$x$ is coloured with $i$} simply to mean that $c\left(x\right)=i$.

In this article we will be concerned with these partition relations exclusively in the case $n=1$. Clearly if $\beta\rightarrow\left(\alpha_i\right)^n_{i\in\kappa}$ and $\gamma>\beta$ then $\gamma\rightarrow\left(\alpha_i\right)^n_{i\in\kappa}$, and similarly for the topological relation. Thus it is sensible to make the following definition. 

\begin{definition}
Let $\alpha_i$ be an ordinal for each $i\in\kappa$.

We define the \emph{non-topological pigeonhole number} $P^{ord}\left(\alpha_i\right)_{i\in\kappa}$ to be the least ordinal $\beta$ such that $\beta\rightarrow\left(\alpha_i\right)^1_{i\in\kappa}$, and the \emph{topological pigeonhole number} $P^{top}\left(\alpha_i\right)_{i\in\kappa}$ to be the least ordinal $\beta$ such that $\beta\rightarrow\left(top\,\alpha_i\right)^1_{i\in\kappa}$.

We extend the usual ordering on the ordinals to include $\infty$ as a maximum. If there is no ordinal $\beta$ such that $\beta\rightarrow\left(top\,\alpha_i\right)^1_{i\in\kappa}$, then we say that $P^{top}\left(\alpha_i\right)_{i\in\kappa}$ does not exist and write $P^{top}\left(\alpha_i\right)_{i\in\kappa}=\infty$.
\end{definition}

Thus for example if $n_1,n_2,\dots,n_k\in\omega\setminus\left\{0\right\}$, then
\[P^{ord}\left(n_1,n_2,\dots,n_k\right)=P^{top}\left(n_1,n_2,\dots,n_k\right)=\sum_{i=1}^k\left(n_i-1\right)+1.\]

Note that $P^{top}\left(\left(\alpha_i\right)_{i\in\kappa},\left(1\right)_\lambda\right)=P^{top}\left(\alpha_i\right)_{i\in\kappa}$ for any cardinal $\lambda$, and that for fixed $\kappa$, $P^{top}\left(\alpha_i\right)_{i\in\kappa}$ is a monotonically increasing function of $\left(\alpha_i\right)_{i\in\kappa}$ (pointwise).

\subsection{The Cantor--Bendixson rank of an ordinal}

Central to the study of ordinal topologies are the notions of Cantor--Bendixson derivative and rank.

\begin{definition}
Let $X$ be a topological space. The \emph{Cantor--Bendixson derivative} $X^\prime$ of $X$ is defined by
\[X^\prime=X\setminus\left\{x\in X:x\text{ is isolated}\right\}.\]

The iterated derivatives of $X$ are defined for $\gamma$ an ordinal by
\begin{enumerate}
\item
$X^{\left(0\right)}=X$,
\item
$X^{\left(\gamma+1\right)}=\left(X^{\left(\gamma\right)}\right)^\prime$, and
\item
$X^{\left(\gamma\right)}=\bigcap_{\delta<\gamma}X^{\left(\delta\right)}$ when $\gamma$ is a non-zero limit.
\end{enumerate}
\end{definition}

Note that if $Y\subseteq X$ then $Y^\prime\subseteq X^\prime$ and hence by transfinite induction $Y^{\left(\gamma\right)}\subseteq X^{\left(\gamma\right)}$ for all ordinals $\gamma$.

\begin{definition}
Let $x$ be an ordinal. Then there is a sequence of ordinals $\gamma_1>\gamma_2>\dots>\gamma_n$ and $m_1,m_2,\dots,m_n\in\omega\setminus\left\{0\right\}$ such that
\[x=\omega^{\gamma_1}\cdot m_1+\omega^{\gamma_2}\cdot m_2+\cdots+\omega^{\gamma_n}\cdot m_n\]
(the \emph{Cantor normal form} of $x$). The \emph{Cantor--Bendixson rank} of $x$ is defined by
\[
\operatorname{CB}\left(x\right)=
\begin{cases}
\gamma_n,&\text{if $x>0$}\\
0,&\text{if $x=0$}.
\end{cases}
\]

This defines a function $\operatorname{CB}:\omega^\beta\rightarrow\beta$ for each non-zero ordinal $\beta$.
\end{definition}

The relationship between these two notions is given by the following simple result, which provides us with an alternative definition of Cantor--Bendixson rank that is entirely topological (though the two definitions need not agree on proper subspaces of ordinals).

\begin{lemma}\label{cantorbendixson}
Let $\alpha$ be an ordinal endowed with the order topology, and let $x\in\alpha$. Then the Cantor--Bendixson rank of $x$ is the greatest ordinal $\gamma$ such that $x\in \alpha^{\left(\gamma\right)}$.
\end{lemma}

In other words, the Cantor--Bendixson rank of an ordinal is the ordinal number of iterated derivatives required to make that point ``disappear''.

\begin{proof}
This is straightforward to prove by induction on the Cantor--Bendixson rank of $x$.
\end{proof}

\subsection{Biembeddability of ordinals}

The notion of biembeddability is a weakening of the notion of homeomorphism that is useful for simplifying the calculation of topological pigeonhole numbers.

\begin{definition}
Let $X$ and $Y$ be topological spaces. We say that $X$ and $Y$ are \emph{biembeddable}, and write $X\approxeq Y$, if and only if $X$ is homeomorphic to a subspace of $Y$ and $Y$ is homeomorphic to a subspace of $X$.
\end{definition}

Clearly $\approxeq$ is an equivalence relation. Its relevance is given by the following easy result.

\begin{lemma}\label{pseudorelevant}
If $Y\approxeq\widetilde Y$ and $X_i\approxeq\widetilde X_i$ for all $i\in\kappa$, then $Y\rightarrow\left(top\,X_i\right)^n_{i\in\kappa}$ if and only if $\widetilde Y\rightarrow\left(top\,\widetilde X_i\right)^n_{i\in\kappa}$.\qed
\end{lemma}

We will now classify the ordinals up to biembeddability, beginning with a positive result.

\begin{lemma}\label{pseudopositive}
Let $\gamma$, $m$ and $\delta$ be non-zero ordinals with $m\in\omega$ and $\delta<\omega^\gamma$. Then
\[\omega^\gamma\cdot m+1\approxeq\omega^\gamma\cdot m+\delta.\]
\end{lemma}

\begin{proof}
Clearly $\omega^\gamma\cdot m+1$ is homeomorphic to a subspace of $\omega^\gamma\cdot m+\delta$, so it is enough to show that $\omega^\gamma\cdot m+\delta$ is homeomorphic to a subspace of $\omega^\gamma\cdot m+1$. In fact, we show that $\omega^\gamma\cdot m+1+\delta+1$ is homeomorphic to $\omega^\gamma\cdot m+1$, which is sufficient.

Now if $\alpha$ and $\beta$ are successor ordinals, then $\alpha+\beta$ is homeomorphic to the topological disjoint union of $\alpha$ and $\beta$, and thus $\alpha+\beta\cong\beta+\alpha$. Hence $\omega^\gamma\cdot m+1+\delta+1\cong\delta+1+\omega^\gamma\cdot m+1=\omega^\gamma\cdot m+1$ since $\delta<\omega^\gamma$.
\end{proof}

Lemmas \ref{pseudorelevant} and \ref{pseudopositive} have the following immediate consequence, which will be useful.

\begin{proposition}\label{pseudosummary}
Let $\kappa$ be a cardinal and let $\alpha_i$ be an ordinal for each $i\in\kappa$. Suppose that for some non-zero ordinal $\gamma$ and some $m\in\omega\setminus\left\{0\right\}$ we have
\[\omega^\gamma\cdot m+1\leq P^{top}\left(\alpha_i\right)_{i\in\kappa}<\omega^\gamma\cdot\left(m+1\right).\] Then in fact
\[P^{top}\left(\alpha_i\right)_{i\in\kappa}=\omega^\gamma\cdot m+1.\pushQED{\qed}\qedhere\popQED\]
\end{proposition}

We now show that Lemma \ref{pseudopositive} is best possible. The proof makes use of Cantor--Bendixson derivatives.

\begin{proposition}\label{pseudonegative}
Let $\gamma$, $m$ and $\delta$ be non-zero ordinals with $m\in\omega$ and $\delta<\omega^\gamma$. Then
\begin{enumerate}
\item
$\omega^\gamma\cdot m+1\not\approxeq\omega^\gamma\cdot m$;
\item\label{pseudolimitcase}
$\omega^\gamma\not\approxeq\delta$; and
\item\label{pseudodifficultcase}
$\omega^\gamma\cdot\left(m+1\right)\not\approxeq\omega^\gamma\cdot m+1$.
\end{enumerate}
\end{proposition}

\begin{proof}
\begin{enumerate}
\item
By Lemma \ref{cantorbendixson}, we have $\left|\left(\omega^\gamma\cdot m+1\right)^{\left(\gamma\right)}\right|=m$ while $\left|\left(\omega^\gamma\cdot m\right)^{\left(\gamma\right)}\right|=m-1$. Therefore no subspace of $\omega^\gamma\cdot m$ can be homeomorphic to $\omega^\gamma\cdot m+1$.
\item\label{pseudolimitcaseproof}
If $\gamma=\eta+1$, then by Lemma \ref{cantorbendixson}, $\left(\omega^\gamma\right)^{\left(\eta\right)}$ is infinite while $\delta^{\left(\eta\right)}$ is finite (or empty). If $\gamma$ is a limit ordinal, then by Lemma \ref{cantorbendixson}, $\left(\omega^\gamma\right)^{\left(\eta\right)}\neq\emptyset$ for all $\eta<\gamma$ while $\delta^{\left(\eta\right)}=\emptyset$ for some $\eta<\gamma$. In either case no subspace of $\delta$ can be homeomorphic to $\omega^\gamma$.
\item
Let $X=\omega^\gamma\cdot\left(m+1\right)$. By Lemma \ref{cantorbendixson}, $X$ has the following two properties: firstly, $\left|X^{\left(\gamma\right)}\right|=m$; and secondly, $X$ has a closed subset $Z$ with $Z\cap X^{\left(\gamma\right)}=\emptyset$ and $Z\cong\omega^\gamma$, namely $Z=\left[\omega^\gamma\cdot m+1,\omega^\gamma\cdot\left(m+1\right)\right)$.

Suppose then that $Y$ is a subspace of $\omega^\gamma\cdot m+1$ with $\left|Y^{\left(\gamma\right)}\right|=m$, and that $W$ is a closed subset of $Y$ with $W\cap Y^{\left(\gamma\right)}=\emptyset$. We show that $W\ncong\omega^\gamma$, which suffices. Since $\left|Y^{\left(\gamma\right)}\right|=m=\left|\left(\omega^\gamma\cdot m+1\right)^{\left(\gamma\right)}\right|$, we must have $Y^{\left(\gamma\right)}=\left(\omega^\gamma\cdot m+1\right)^{\left(\gamma\right)}=\left\{\omega^\gamma,\omega^\gamma\cdot 2,\dots,\omega^\gamma\cdot m\right\}$. Therefore since $W$ is closed, for each $i\in\left\{0,\dots,m-1\right\}$ there exists $x_i\in\left[\omega^\gamma\cdot i+1,\omega^\gamma\cdot\left(i+1\right)\right)$ such that $W\cap\left(x_i,\omega^\gamma\cdot\left(i+1\right)\right)=\emptyset$. It follows that $W$ is homeomorphic to the disjoint union of a finite number of subspaces of $\zeta$ for some $\zeta<\omega^\gamma$. The argument of part \ref{pseudolimitcaseproof} then shows that $W\ncong\omega^\gamma$.\qedhere
\end{enumerate}
\end{proof}

Although we will not make use of the following result, it is of interest nonetheless. It is an immediate consequence of Lemma \ref{pseudopositive} and Proposition \ref{pseudonegative}.

\begin{corollary}[Classification of ordinals up to biembeddability]
Two ordinals $\alpha\leq\beta$ are biembeddable if and only if either $\alpha=\beta=\omega^\gamma\cdot m$ for some ordinal $\gamma$ and some $m\in\omega$, or $\omega^\gamma\cdot m+1\leq\alpha\leq\beta<\omega^\gamma\cdot\left(m+1\right)$ for some non-zero ordinal $\gamma$ and some $m\in\omega\setminus\left\{0\right\}$.\qed
\end{corollary}

\subsection{The order-homeomorphism relation}

It will sometimes be necessary to make use of the order structure of a homeomorphic copy of an ordinal, for which the notion of order-homeomorphism will be useful.

\begin{definition}
Let $X$ and $Y$ be sets of ordinals. We say that $X$ and $Y$ are \emph{order-homeomorphic} if and only if there is a order-preserving homeomorphism from $X$ to $Y$.
\end{definition}

Note that a set of ordinals is always assumed to have the subspace topology induced from the order topology on any ordinal containing it, which need not coincide with the order topology on the set itself. In fact it is easy to see that these topologies coincide if and only if $X$ is order-homeomorphic to its order type.

The following simple result provides us with an equivalent criterion for this scenario in terms of closed sets. The proof is left as an exercise.

\begin{proposition}\label{ordertypehomeomorphic}
Let $X$ be a set of ordinals, and let $\eta$ be the least ordinal with $X\subseteq\eta$. Then $X$ is a closed subset of $\eta$ if and only if $X$ is order-homeomorphic to its order type.
\end{proposition}

At one point, we will be able to make use of the notion of order-homeomorphism because of the following property held by certain ordinals.

\begin{definition}
Let $\alpha$ be an ordinal. We say that $\alpha$ is \emph{order-reinforcing} if and only if, whenever $X$ is a set of ordinals with $X\cong\alpha$, there is a subset $Y\subseteq X$ such that $Y$ is order-homeomorphic to $\alpha$.
\end{definition}

Baumgartner \cite[Theorem 0.2]{baumgartner} showed that every countable ordinal of the form $\omega^\gamma+1$ or $\omega^\gamma$ is order-reinforcing. We now extend this result.

\begin{theorem}\label{orderreinforcing}
Let $\gamma$ be a non-zero ordinal and let $m\in\omega\setminus\left\{0\right\}$. Then
\begin{enumerate}
\item\label{orderreinforcingsuccessor}
$\omega^\gamma\cdot m+1$ is order-reinforcing; and
\item
$\omega^\gamma$ is order-reinforcing.
\end{enumerate}
\end{theorem}

Baumgartner's proof for ordinals of the form $\omega^\gamma+1$ is also valid for uncountable ordinals of this form, and our proof of part \ref{orderreinforcingsuccessor} is almost identical. Baumgartner's proof for ordinals of the form $\omega^\gamma$ is valid for uncountable ordinals of this form providing they have countable cofinality, so we provide a new proof to cover the remaining case.

In the proof, given a topological space $A$ and a subset $B\subseteq A$, we write $\operatorname{cl}_A\left(B\right)$ for the closure of $B$ in $A$.

\begin{proof}
\begin{enumerate}
\item
Let $\alpha=\omega^\gamma\cdot m+1$, let $X$ be a set of ordinals with $X\cong\alpha$, and let $\eta$ be the least ordinal with $X\subseteq\eta$. Then $X$ is compact and therefore a closed subset of the Hausdorff space $\eta$. So by Proposition \ref{ordertypehomeomorphic}, $X$ is order-homeomorphic to its order type. This order type must be at least $\alpha$ in order for $\left|X^{\left(\gamma\right)}\right|=m$. Hence we may take $Y$ to be the initial segment of $X$ of order type $\alpha$.
\item
Let $\alpha=\omega^\gamma$. Baumgartner's proof covers the case in which $\alpha$ has countable cofinality, so assume that $\alpha$ has uncountable cofinality.

Let $X$ be a set of ordinals with $X\cong\alpha$, and let $\eta$ be the least ordinal with $X\subseteq\eta$. Then $X$ is not compact and is therefore not a closed subset of the compact space $\eta+1$. So we may let $x$ be the minimal element of $\operatorname{cl}_{\eta+1}\left(X\right)\setminus X$. Let $Z=X\cap\left[0,x\right)$, so that $Z$ is a closed cofinal subset of $\left[0,x\right)$. Then by Proposition \ref{ordertypehomeomorphic}, $Z$ is order-homeomorphic to its order type, say the ordinal $\beta$. Observe now that $Z$ is a closed open subset of $X$ but is not compact. We claim that any closed open subset of $\alpha$ that is not compact must be homeomorphic to $\alpha$. From this it follows that $\beta\cong\alpha$ and hence $\beta\geq\alpha$ by Proposition \ref{pseudonegative} part \ref{pseudolimitcase}. Hence we may take $Y$ to be the initial segment of $Z$ of order type $\alpha$.

To prove the claim, suppose $W$ is a closed open subset of $\alpha$ that is not compact. Then $W$ and $\alpha\setminus W$ are both closed subsets of $\alpha$, but they are disjoint and so cannot both be club in $\alpha$. Now any closed bounded subset of $\alpha$ is compact, so it must be that $W$ is unbounded in $\alpha$ while $\alpha\setminus W$ is bounded. It follows that $W$ is a closed subset of $\alpha$ of order type $\alpha$, and so $W\cong\alpha$ by Proposition \ref{ordertypehomeomorphic}. This proves the claim, which completes the proof.\qedhere
\end{enumerate}
\end{proof}

Although we will not make use of the fact, it is interesting to note that this result is best possible for infinite ordinals, as we now show.

\begin{corollary}[Classification of order-reinforcing ordinals]
An ordinal $\alpha$ is order-reinforcing if and only if either $\alpha$ is finite, or $\alpha=\omega^\gamma\cdot m+1$ for some non-zero ordinal $\gamma$ and some $m\in\omega\setminus\left\{0\right\}$, or $\alpha=\omega^\gamma$ for some non-zero ordinal $\gamma$.
\end{corollary}

\begin{proof}
The ``if'' statement follows from Theorem \ref{orderreinforcing} and the fact that every finite ordinal is order-reinforcing.

For the ``only if'' statement, if $\alpha$ is infinite then we may write $\alpha=\omega^\gamma\cdot m+\delta$ with $\gamma$ a non-zero ordinal, $m\in\omega\setminus\left\{0\right\}$ and $\delta<\omega^\gamma$. Assume that $\alpha$ does not have one of the given forms, so that either $\delta>1$, or $\delta=0$ and $m>1$. If $\delta>1$, then by Lemma \ref{pseudopositive} we may take $X$ to be a subspace of $\omega^\gamma\cdot m+1$ with $X\cong\alpha$. If $\delta=0$ and $m>1$, then we may take $X=\left(\omega^\gamma\cdot m+1\right)\setminus\left\{\omega^\gamma\right\}$. In either case $X$ is a witness to the fact that $\alpha$ is not order-reinforcing.
\end{proof}

\subsection{The Milner--Rado sum of ordinals}

In our calculation of the topological pigeonhole numbers of finite sequences of countable ordinals, we will make use of two different binary operations on ordinals: the natural sum, due to Hessenberg, and the Milner--Rado sum.

\begin{definition}
Let $\alpha$ and $\beta$ be ordinals. Then we may choose a sequence of ordinals $\gamma_1>\gamma_2>\dots>\gamma_n$ and $l_1,l_2,\dots,l_n,m_1,m_2,\dots,m_n\in\omega$ such that
\[\alpha=\omega^{\gamma_1}\cdot l_1+\omega^{\gamma_2}\cdot l_2+\cdots+\omega^{\gamma_n}\cdot l_n\]
and
\[\beta=\omega^{\gamma_1}\cdot m_1+\omega^{\gamma_2}\cdot m_2+\cdots+\omega^{\gamma_n}\cdot m_n.\]
We define the \emph{natural sum} of $\alpha$ and $\beta$ by
\[\alpha\#\beta=\omega^{\gamma_1}\cdot\left(l_1+m_1\right)+\omega^{\gamma_2}\cdot\left(l_2+m_2\right)+\cdots+\omega^{\gamma_n}\cdot\left(l_n+m_n\right).\]
\end{definition}

\begin{definition}
Let $\alpha$ and $\beta$ be ordinals. Then we define the \emph{Milner--Rado sum} of $\alpha$ and $\beta$, denoted by $\alpha\odot\beta$, to be the least ordinal $\delta$ such that if $\widetilde\alpha<\alpha$ and $\widetilde\beta<\beta$ then $\delta\neq\widetilde\alpha\#\widetilde\beta$.
\end{definition}

Milner and Rado \cite{milnerrado} observed that if $\zeta>\alpha\odot\beta$, $\widetilde\alpha<\alpha$ and $\widetilde\beta<\beta$, then $\zeta\neq\widetilde\alpha\#\widetilde\beta$.

They also observed that both $\#$ and $\odot$ are commutative and associative, and so brackets may be omitted when three or more ordinals are summed. Notice that if $\alpha_1,\alpha_2,\dots,\alpha_k$ are ordinals, then $\alpha_1\odot\alpha_2\odot\cdots\odot\alpha_k$ is simply the least ordinal $\delta$ such that if $\widetilde\alpha_i<\alpha_i$ for all $i\in\left\{1,2,\dots,k\right\}$ then $\delta\neq\widetilde\alpha_1\#\widetilde\alpha_2\#\cdots\#\widetilde\alpha_k$.

As part of their computation of the non-topological pigeonhole numbers, Milner and Rado computed the Milner--Rado sum of an arbitrary finite sequence of ordinals.

\begin{theorem}[Milner--Rado]\label{milnerrado}
Let $\alpha_1,\alpha_2,\dots,\alpha_k$ be non-zero ordinals. We may choose a sequence of ordinals $\gamma_1>\gamma_2>\dots>\gamma_N$ and, for each $i\in\left\{1,2,\dots,k\right\}$, $m_{i1},m_{i2},\dots,m_{in_i}\in\omega$ such that for each $i\in\left\{1,2,\dots,k\right\}$, $m_{in_i}>0$ and
\[\alpha_i=\omega^{\gamma_1}\cdot m_{i1}+\omega^{\gamma_2}\cdot m_{i2}+\cdots+\omega^{\gamma_{n_i}}\cdot m_{in_i}.\]

Let $n=\operatorname{min}\left\{n_1,n_2,\dots,n_k\right\}$ and let $s_j=\sum_{i=1}^km_{ij}$ for each $j\in\left\{1,2,\dots,n\right\}$. Finally let $t=\left|\left\{i\in\left\{1,2,\dots,k\right\}:n_i=n\right\}\right|$. Then
\begin{align*}
P^{ord}\left(\alpha_1,\alpha_2,\dots,\alpha_k\right)&=\alpha_1\odot\alpha_2\odot\cdots\odot\alpha_k\\
&=\omega^{\gamma_1}\cdot s_1+\omega^{\gamma_2}\cdot s_2+\cdots+\omega^{\gamma_{n-1}}\cdot s_{n-1}+\omega^{\gamma_n}\cdot\left(s_n-t+1\right).
\end{align*}
\end{theorem}

This result permits us to reduce the calculation of certain topological pigeonhole numbers to corresponding non-topological pigeonhole numbers as in Theorem \ref{link}.

\subsection{A simple example}

We conclude this section by proving the following special case of Theorem \ref{link} in order to illustrate the character of many later proofs.

\begin{proposition}\label{simple}
Let $k$ be a positive integer. Then
\[P^{top}\left(\omega+1\right)_k=\omega^k+1.\]
\end{proposition}

The main idea in the proof is the following result, which says that any finite colouring of $\omega^n$ is in some sense similar to a colouring which is constant on ordinals of the same Cantor--Bendixson rank.

\begin{lemma}\label{simplelemma}
Let $k$ and $n$ be positive integers and let $c:\omega^n\rightarrow k$. Then there is some subset $X\subseteq\omega^n$ such that $X\cong\omega^n$ and $c$ is constant on $X^{\left(i\right)}\setminus X^{\left(i+1\right)}$ for each $i\in n$.
\end{lemma}

\begin{proof}
The proof is by induction on $n$. The case $n=1$ is simply the ordinary pigeonhole principle $\omega\rightarrow\left(top\,\omega\right)^1_k$, so assume $n>1$. Consider first the restriction of $c$ to $\left\{\omega\cdot\alpha:\alpha\in\omega^{n-1}\right\}$. By the inductive hypothesis, passing to a subset we may assume that $c$ is constant on $\left(\omega^n\right)^{\left(i\right)}\setminus\left(\omega^n\right)^{\left(i+1\right)}$ for each $i\in n\setminus\left\{0\right\}$. By considering the restriction of $c$ to $\left[\omega^{n-1}\cdot m+1,\omega^{n-1}\cdot\left(m+1\right)\right]$ for each $m\in\omega$, we may likewise assume that $c$ is constant on $\left(\omega^n\setminus\left(\omega^n\right)^\prime\right)\cap\left[\omega^{n-1}\cdot m+1,\omega^{n-1}\cdot\left(m+1\right)\right]$ for each $m\in\omega$, taking the value $c_m$, say. To finish, simply choose an infinite subset $S\subseteq\omega$ such that $c_l=c_m$ for all $l,m\in S$, and take $X$ to be\phantom\qedhere
\[\bigcup_{m\in S}\left[\omega^{n-1}\cdot m+1,\omega^{n-1}\cdot\left(m+1\right)\right].\pushQED{\qed}\qedhere\popQED\]
\end{proof}

\begin{proof}[Proof of Proposition \ref{simple}]
To see that $\omega^k\nrightarrow\left(top\,\omega+1\right)^1_k$, simply colour each $x\in\omega^k$ with colour $\operatorname{CB}\left(x\right)$, and observe that each colour class is discrete.

To see that $\omega^k+1\rightarrow\left(top\,\omega+1\right)^1_k$, let $c:\omega^k+1\rightarrow k$. Choose $X\subseteq\omega^k$ as in Lemma \ref{simplelemma}, and let $Y=X\cup\left\{\omega^k\right\}$. Since $Y^{\left(k\right)}$ is simply the singleton $\left\{\omega^k\right\}$, we in fact have that $c$ is constant on $Y^{\left(i\right)}\setminus Y^{\left(i+1\right)}$ for each $i\in k+1$. By the finite pigeonhole principle $k+1\rightarrow\left(2\right)^1_k$, it follows that $c$ is constant on $\left(Y^{\left(i\right)}\setminus Y^{\left(i+1\right)}\right)\cup\left(Y^{\left(j\right)}\setminus Y^{\left(j+1\right)}\right)$ for some distinct $i,j\in k+1$, a set which is easily seen to contain a homeomorphic copy of $\omega+1$.
\end{proof}

The key idea to take from this proof is the importance of colourings of the form $\widetilde c\circ\operatorname{CB}$ for some $\widetilde c:k\rightarrow k$. The negative relation was proved using a counterexample of this form. The positive relation was proved by showing in the Lemma that any colouring must be similar to some colouring of this form, and applying the pigeonhole principle to $k+1$. The proof of Theorem \ref{link} will be similar, with this use of the Lemma and the pigeonhole principle replaced by Weiss's result.

\section{Statement of the principle}

We now state the main theorem of this article. Although it may not be necessary to go through the details of every case at this stage, they are included here for completeness and for reference. Our main breakthrough is given in case \ref{finiteofcountable}, and includes Theorem \ref{link} as a special case.

Observe first that if $\alpha_r=0$ for some $r\in\kappa$, then $P^{top}\left(\alpha_i\right)_{i\in\kappa}=0$, and if $I\subseteq\kappa$ with $\alpha_i=1$ for all $i\in I$, then $P^{top}\left(\alpha_i\right)_{i\in\kappa}=P^{top}\left(\alpha_i\right)_{i\in\kappa\setminus I}$. Thus it is sufficient to consider the cases in which $\alpha_i\geq 2$ for all $i\in\kappa$.

\begin{theorem}[The topological pigeonhole principle for ordinals]\label{principle}
Let $\kappa$ be a cardinal, and let $\alpha_i$ be an ordinal with $\alpha_i\geq 2$ for all $i\in\kappa$.
\begin{enumerate}
\item\label{infinityprovablecase}
If $\alpha_r\geq\omega_1+1$ and $\alpha_s\geq\omega+1$ for some distinct $r,s\in\kappa$, then $P^{top}\left(\alpha_i\right)_{i\in\kappa}=\infty$.
\item\label{lopsidedcase}
If $\alpha_r\geq\omega_1+1$ for some $r\in\kappa$ and $\alpha_i\leq\omega$ for all $i\in\kappa\setminus\left\{r\right\}$:
\begin{enumerate}
\item
if $\kappa\geq\aleph_0$:
\begin{enumerate}
\item
if $\alpha_r$ is a not a power of $\omega$, then $P^{top}\left(\alpha_i\right)_{i\in\kappa}=\alpha_r\cdot\kappa^+$;
\item
if $\alpha_r$ is a power of $\omega$:
\begin{enumerate}
\item
if $\operatorname{cf}\left(\alpha_r\right)>\kappa$, then $P^{top}\left(\alpha_i\right)_{i\in\kappa}=\alpha_r$;
\item
if $\aleph_0<\operatorname{cf}\left(\alpha_r\right)\leq\kappa$, then $P^{top}\left(\alpha_i\right)_{i\in\kappa}=\alpha_r\cdot\kappa^+$;
\item
if $\operatorname{cf}\left(\alpha_r\right)=\aleph_0$, then we may write $\alpha_r=\omega^\beta$ and $\beta=\gamma+\omega^\delta$ with $\delta$ not a limit ordinal of uncountable cofinality; then
\begin{itemize}
\item
if $\delta<\kappa^+$, then $P^{top}\left(\alpha_i\right)_{i\in\kappa}=\alpha_r\cdot\kappa^+$;
\item
if $\delta>\kappa^+$, then $P^{top}\left(\alpha_i\right)_{i\in\kappa}=\alpha_r$;
\end{itemize}
\end{enumerate}
\end{enumerate}
\item
if $\kappa<\aleph_0$ and $\alpha_s=\omega$ for some $s\in\kappa\setminus\left\{r\right\}$:
\begin{enumerate}
\item
if $\alpha_r$ is a power of $\omega$, then $P^{top}\left(\alpha_i\right)_{i\in\kappa}=\alpha_r$;
\item
if $\alpha_r$ is not a power of $\omega$, then $P^{top}\left(\alpha_i\right)_{i\in\kappa}=\alpha_r\cdot\omega$;
\end{enumerate}
\item
if $\kappa<\aleph_0$ and $\alpha_i<\omega$ for all $i\in\kappa\setminus\left\{r\right\}$:
\begin{enumerate}
\item
if $\alpha_r$ is a power of $\omega$ or $\kappa=1$, then $P^{top}\left(\alpha_i\right)_{i\in\kappa}=\alpha_r$;
\item\label{uncountablewithmultiplescase}
if $\kappa>1$ and $\alpha_r$ is not a power of $\omega$, then $\omega^\beta\cdot m+1\leq\alpha_r\leq\omega^\beta\cdot\left(m+1\right)$ for some ordinal $\beta$ and some $m\in\omega\setminus\left\{0\right\}$; then
\[P^{top}\left(\alpha_i\right)_{i\in\kappa}=\omega^\beta\cdot\left(\sum_{i\in\kappa\setminus\left\{r\right\}}\left(\alpha_i-1\right)+m\right)+1.\]
\end{enumerate}
\end{enumerate}
\item\label{independencecase}
If $\alpha_i\leq\omega_1$ for all $i\in\kappa$ and $\alpha_r,\alpha_s=\omega_1$ for some distinct $r,s\in\kappa$, then the value of $P^{top}\left(\alpha_i\right)_{i\in\kappa}$ is independent of $\mathsf{ZFC}$ in the following sense.

Write ``$P_\kappa=x$'' for the statement, ``$\kappa$ is a cardinal, and for all sequences $\left(\alpha_i\right)_{i\in\kappa}$ of ordinals, if $2\leq\alpha_i\leq\omega_1$ for all $i\in\kappa$ and $\alpha_r,\alpha_s=\omega_1$ for some distinct $r,s\in\kappa$, then $P^{top}\left(\alpha_i\right)_{i\in\kappa}=x$''. Likewise for ``$P_\kappa\geq x$''.

Firstly,
\[\text{``for all cardinals $\kappa\geq 2$, $P_\kappa\geq\operatorname{max}\left\{\omega_2,\kappa^+\right\}$''}\]
is a theorem of $\mathsf{ZFC}$. Secondly, if $\mathsf{ZFC}$ is consistent, then so is
\[\mathsf{ZFC}+\text{``for all cardinals $\kappa\geq 2$, $P_\kappa=\infty$''}.\]
Thirdly, if $\mathsf{ZFC}+$``there exists a supercompact cardinal'' is consistent, then so is
\[\mathsf{ZFC}+\text{``for all cardinals $\kappa\geq 2$, $P_\kappa=\operatorname{max}\left\{\omega_2,\kappa^+\right\}$''}.\]
Moreover, some large cardinal assumption is required, since $\mathsf{ZFC}+$``there exists a Mahlo cardinal'' is consistent if and only if
\[\mathsf{ZFC}+\text{``$\omega_2\rightarrow\left(top\,\omega_1\right)^1_2$''}\]
is consistent.
\item\label{justoneomega1}
If $\alpha_r=\omega_1$ for some $r\in\kappa$ and $\alpha_i<\omega_1$ for all $i\in\kappa\setminus\left\{r\right\}$, then $P^{top}\left(\alpha_i\right)_{i\in\kappa}=\operatorname{max}\left\{\omega_1,\kappa^+\right\}$.
\item\label{infiniteofcountablecase}
If $\alpha_i<\omega_1$ for all $i\in\kappa$ and $\kappa\geq\aleph_0$, then $P^{top}\left(\alpha_i\right)_{i\in\kappa}=\kappa^+$.
\item\label{finiteofcountable}
If $\alpha_i<\omega_1$ for all $i\in\kappa$ and $\kappa<\aleph_0$:
\begin{enumerate}
\item
if $\alpha_i<\omega$ for all $i\in\kappa$, then
\[P^{top}\left(\alpha_i\right)_{i\in\kappa}=\sum_{i\in\kappa}\left(\alpha_i-1\right)+1;\]
\item\label{finiteofcountablepower}
if $\alpha_r$ is a power of $\omega$ for some $r\in\kappa$, then
\[P^{top}\left(\alpha_i\right)_{i\in\kappa}=\omega^{\beta_0\odot\beta_1\odot\cdots\odot\beta_{\kappa-1}},\]
where for each $i\in\kappa$, $\beta_i$ is minimal subject to the condition that $\alpha_i\leq\omega^{\beta_i}$;
\item\label{finiteofcountableother}
if $\alpha_i$ is not a power of $\omega$ for any $i\in\kappa$ and $\alpha_r\geq\omega$ for some $r\in\kappa$, then for each $i\in\kappa$ we can find an ordinal $\beta_i$ and $m_i\in\omega\setminus\left\{0\right\}$ such that either $\alpha_i=m_i$ and $\beta_i=0$, or $\omega^{\beta_i}\cdot m_i+1\leq\alpha_i\leq\omega^{\beta_i}\cdot\left(m_i+1\right)$ and $\beta_i>0$; then:
\begin{enumerate}
\item\label{distinguishingcase}
if there exists $s\in\kappa$ such that $\alpha_s=\omega^{\beta_s}\cdot\left(m_s+1\right)$, $\operatorname{CB}\left(\beta_s\right)\leq\operatorname{CB}\left(\beta_i\right)$ for all $i\in\kappa$, and $m_i=1$ for all $i\in\kappa\setminus\left\{s\right\}$, then
\[P^{top}\left(\alpha_i\right)_{i\in\kappa}=\omega^{\beta_0\#\beta_1\#\cdots\#\beta_{\kappa-1}}\cdot\left(m_s+1\right);\]
\item
otherwise,
\[P^{top}\left(\alpha_i\right)_{i\in\kappa}=\omega^{\beta_0\#\beta_1\#\cdots\#\beta_{\kappa-1}}\cdot\left(\sum_{i\in\kappa}\left(m_i-1\right)+1\right)+1.\]
\end{enumerate}
\end{enumerate}
\end{enumerate}
\end{theorem}

We prove this result in a case-by-case fashion, as follows. Case \ref{infinityprovablecase} has a simple proof, which we give in Proposition \ref{infinityprovable}. Case \ref{lopsidedcase} has many subcases, each of which has a relatively straightforward proof; we reformulate these subcases using an elementary argument in Lemma \ref{lopsidedsimplification}, before proving each one individually in Section \ref{principleproof}. Case \ref{independencecase} is easy to deduce from results of others, which we do in Section \ref{independenceresults}. Cases \ref{justoneomega1} and \ref{infiniteofcountablecase} have simple proofs involving stationary sets, which we give in Section \ref{justoneomega1andinfiniteofcountable}. Finally, case \ref{finiteofcountable} has the most new ideas. We provide the key ingredients for the proof in Section \ref{finiteofcountablesection}, before combining them to complete the proof in Section \ref{principleproof}. We describe the key ideas first, including the proof of Theorem \ref{link} in Theorems \ref{linksuccessor} and \ref{linkpowers}.

\section{Proof of the principle}

\subsection{Finite sequences of countable ordinals}\label{finiteofcountablesection}

We begin with case \ref{finiteofcountable} of the principle, including the proof of Theorem \ref{link}. First of all we state Weiss's result, which requires us to introduce some notation.

\begin{definition}
Let $\gamma_1\geq\gamma_2\geq\dots\geq\gamma_n$ be ordinals and $S\subseteq\left\{1,2,\dots,n\right\}$, say $S=\left\{s_1,s_2,\dots,s_l\right\}$ with $s_1<s_2<\dots<s_l$. Then we write
\[\sum_{i\in S}\omega^{\gamma_i}=\omega^{\gamma_{s_1}}+\omega^{\gamma_{s_2}}+\cdots+\omega^{\gamma_{s_l}}\]
and
\[
\left(\omega^{\omega^{\gamma_1}+\omega^{\gamma_2}+\cdots+\omega^{\gamma_n}}\right)_S=
\begin{cases}
\omega^{\sum_{i\in S}\omega^{\gamma_i}},&\text{if $S\neq\emptyset$}\\
0,&\text{if $S=\emptyset$}.
\end{cases}
\]
\end{definition}

Weiss's result is our key tool for proving positive relations in this section, and was first published by Baumgartner \cite[Theorem 2.3]{baumgartner}.

\begin{theorem}[Weiss]
Let $\gamma_1\geq\gamma_2\geq\dots\geq\gamma_n$ be countable ordinals, let \[\beta=\omega^{\omega^{\gamma_1}+\omega^{\gamma_2}+\cdots+\omega^{\gamma_n}},\]
and let $c:\beta\rightarrow 2$. Then there exists $S\subseteq\left\{1,2,\dots,n\right\}$, $X\subseteq c^{-1}\left(\left\{0\right\}\right)$ and $Y\subseteq c^{-1}\left(\left\{1\right\}\right)$ such that $X\cong\beta_S$, $Y\cong\beta_{\left(\left\{1,2,\dots,n\right\}\setminus S\right)}$ and $X$ and $Y$ are both either empty or cofinal in $\beta$.
\end{theorem}

A careful reading of Baumgartner's proof reveals that ``homeomorphic'' can in fact be strengthened to ``order-homeomorphic''. Furthermore, we will be interested in colourings using more than 2 colours. It will therefore be more convenient to use this result in the following form.

\begin{corollary}\label{weisscorollary}
Let $\beta$ be as in Weiss's theorem, let $k$ be a positive integer and let $c:\beta\rightarrow k$. Then there exists a partition of $\left\{1,2,\dots,n\right\}$ into $k$ pieces $S_0,S_1,\dots,S_{k-1}$ and for each $i\in k$ a subset $X_i\subseteq c^{-1}\left(\left\{i\right\}\right)$ such that for all $i\in k$, $X_i$ is order-homeomorphic to $\beta_{S_i}$ and $X_i$ is either empty or cofinal in $\beta$.
\end{corollary}

\begin{proof}
This follows immediately from the ``order-homeomorphic'' version of Weiss's theorem by induction on $k$.
\end{proof}

To prove negative relations we will frequently consider colourings based on those of the form $\widetilde c\circ\operatorname{CB}$ for some $\widetilde c:\beta\rightarrow\kappa$, where $\beta$ is a non-zero ordinal. The following result is our key tool for analysing these colourings.

\begin{proposition}\label{rankcolouring}
Let $\alpha$ and $\eta$ be ordinals. Let $Y$ be a set of ordinals of order type $\alpha$, and let $X=\left\{x\in\eta:\operatorname{CB}\left(x\right)\in Y\right\}$. Then $X^{\left(\alpha\right)}=\emptyset$.
\end{proposition}

\begin{proof}
For each $\zeta\leq\alpha$, let $Y_\zeta$ be the initial segment of $Y$ of order type $\zeta$ and let $X_\zeta=\left\{x\in\eta:\operatorname{CB}\left(x\right)\in Y_\zeta\right\}$. It is easy to prove by induction on $\zeta\leq\alpha$ that $X^{\left(\zeta\right)}=X\setminus X_\zeta$. Hence $X^{\left(\alpha\right)}=X\setminus X_\alpha=\emptyset$.
\end{proof}

We can now apply these two tools to prove Theorem \ref{link}, beginning with part \ref{linkpartsuccessor}.

\begin{theorem}\label{linksuccessor}
Let $\alpha_0,\alpha_1,\dots,\alpha_{k-1}\in\omega_1\setminus\left\{0\right\}$. Then
\[P^{top}\left(\omega^{\alpha_0}+1,\omega^{\alpha_1}+1,\dots,\omega^{\alpha_{k-1}}+1\right)=\omega^{\alpha_0\#\alpha_1\#\cdots\#\alpha_{k-1}}+1.\]
\end{theorem}

\begin{proof}
Write $\alpha_0\#\alpha_1\#\cdots\#\alpha_{k-1}=\delta=\omega^{\gamma_1}+\omega^{\gamma_2}+\cdots+\omega^{\gamma_n}$ with $\gamma_1\geq\gamma_2\geq\dots\geq\gamma_n$, and write $\beta=\omega^\delta$.

To see that $\beta\nrightarrow\left(top\,\omega^{\alpha_0}+1,\omega^{\alpha_1}+1,\dots,\omega^{\alpha_{k-1}}+1\right)^1$, first observe that by definition of the natural sum, there is a partition of $\left\{1,2,\dots,n\right\}$ into $k$ pieces $S_0,S_1,\dots,S_{k-1}$ such that for all $i\in k$, $\alpha_i=\sum_{j\in S_i}\omega^{\gamma_j}$. Now define a colouring $c:\beta\rightarrow k$ as follows. For each $i\in k$, set $c\left(x\right)=i$ if and only if
\[\omega^{\gamma_1}+\omega^{\gamma_2}+\cdots+\omega^{\gamma_{j-1}}\leq\operatorname{CB}\left(x\right)<\omega^{\gamma_1}+\omega^{\gamma_2}+\cdots+\omega^{\gamma_j}\]
for some $j\in S_i$. Observe that $c^{-1}\left(\left\{i\right\}\right)=\left\{x\in\beta:\operatorname{CB}\left(x\right)\in Y_i\right\}$ for some set $Y_i$ of ordinals of order type $\alpha_i$. Thus by Proposition \ref{rankcolouring}, $c^{-1}\left(\left\{i\right\}\right)^{\left(\alpha_i\right)}=\emptyset$, whereas $\left(\omega^{\alpha_i}+1\right)^{\left(\alpha_i\right)}=\left\{\omega^{\alpha_i}\right\}$. Hence $c^{-1}\left(\left\{i\right\}\right)$ cannot contain a homeomorphic copy of $\omega^{\alpha_i}+1$.

To see that $\beta+1\rightarrow\left(top\,\omega^{\alpha_0}+1,\omega^{\alpha_1}+1,\dots,\omega^{\alpha_{k-1}}+1\right)^1$, let $c:\beta+1\rightarrow k$. Choose $S_0,S_1,\dots,S_{k-1}\subseteq\left\{1,2,\dots,n\right\}$ and $X_0,X_1,\dots,X_{k-1}\subseteq\beta$ as in Corollary \ref{weisscorollary}. If $\beta_{S_i}>\omega^{\alpha_i}$ for some $i\in k$, then we are done. So we may assume $\beta_{S_i}\leq\omega^{\alpha_i}$ for all $i\in k$. But then we must in fact have $\beta_{S_i}=\omega^{\alpha_i}$ for all $i\in k$, else $\beta<\omega^{\alpha_0\#\alpha_1\#\cdots\#\alpha_{k-1}}$. To finish, suppose $c\left(\beta\right)=j$. Then since $X_j$ is cofinal in $\beta$, $X_j\cup\left\{\beta\right\}$ is a homeomorphic copy of $\omega^{\alpha_j}+1$ in colour $j$.
\end{proof}

The proof of part \ref{linkpartpowers} of Theorem \ref{link} is similar but a little more complicated as it makes use of the Milner--Rado sum. We make use of the fact that $P^{ord}\left(\alpha_0,\alpha_1,\dots,\alpha_{k-1}\right)=\alpha_0\odot\alpha_1\odot\cdots\odot\alpha_{k-1}$ by using the first expression to prove the negative relation and the second expression to prove the positive relation.

\begin{theorem}\label{linkpowers}
Let $\alpha_0,\alpha_1,\dots,\alpha_{k-1}\in\omega_1\setminus\left\{0\right\}$. Then
\[P^{top}\left(\omega^{\alpha_0},\omega^{\alpha_1},\dots,\omega^{\alpha_{k-1}}\right)=\omega^{\alpha_0\odot\alpha_1\odot\cdots\odot\alpha_{k-1}}.\]
\end{theorem}

\begin{proof}
First recall from Theorem \ref{milnerrado} that $P^{ord}\left(\alpha_0,\alpha_1,\dots,\alpha_{k-1}\right)=\alpha_0\odot\alpha_1\odot\cdots\odot\alpha_{k-1}$. Write $\delta$ for their common value, write $\delta=\omega^{\gamma_1}+\omega^{\gamma_2}+\cdots+\omega^{\gamma_n}$ with $\gamma_1\geq\gamma_2\geq\dots\geq\gamma_n$, and write $\beta=\omega^\delta$.

Suppose $\zeta<\beta$. To see that $\zeta\nrightarrow\left(top\,\omega^{\alpha_0},\omega^{\alpha_1},\dots,\omega^{\alpha_{k-1}}\right)^1$, first observe that $\zeta\leq\omega^\eta\cdot m+1$ for some $\eta<\delta$ and some $m\in\omega$, so it is sufficient to consider the case in which $\zeta=\omega^\eta\cdot m+1$. Since $\eta<P^{ord}\left(\alpha_0,\alpha_1,\dots,\alpha_{k-1}\right)$, there is a colouring $\widetilde c:\eta\rightarrow k$ such that for all $i\in k$, the order type of $\widetilde c^{-1}\left(\left\{i\right\}\right)$ is $\widetilde\alpha_i<\alpha_i$. Let $c:\zeta\rightarrow k$ be a colouring with $c\left(x\right)=\widetilde c\left(\operatorname{CB}\left(x\right)\right)$ for all $x\in\zeta\setminus\left\{\omega^\eta,\omega^\eta\cdot 2,\dots,\omega^\eta\cdot m\right\}$ (it doesn't matter how the points $\omega^\eta,\omega^\eta\cdot 2,\dots,\omega^\eta\cdot m$ are coloured). By Proposition \ref{rankcolouring}, $c^{-1}\left(\left\{i\right\}\right)^{\left(\widetilde\alpha_i\right)}\subseteq\left\{\omega^\eta,\omega^\eta\cdot 2,\dots,\omega^\eta\cdot m\right\}$ for all $i\in k$, whereas $\left(\omega^{\alpha_i}\right)^{\left(\widetilde\alpha_i\right)}$ is infinite since $\widetilde\alpha_i<\alpha_i$. Hence $c^{-1}\left(\left\{i\right\}\right)$ cannot contain a homeomorphic copy of $\omega^{\alpha_i}$.

To see that $\beta\rightarrow\left(top\,\omega^{\alpha_0},\omega^{\alpha_1},\dots,\omega^{\alpha_{k-1}}\right)^1$, let $c:\beta\rightarrow k$. Choose $S_0,S_1,\dots,S_{k-1}\subseteq\left\{1,2,\dots,n\right\}$ and $X_0,X_1,\dots,X_{k-1}\subseteq\beta$ as in Corollary \ref{weisscorollary}. If $\beta_{S_i}\geq\omega^{\alpha_i}$ for some $i\in k$, then we are done, so suppose for contradiction that $\beta_{S_i}<\omega^{\alpha_i}$ for all $i\in k$. Write $\widetilde\alpha_i=\sum_{j\in S_i}\omega^{\gamma_i}$, so that $\omega^{\widetilde\alpha_i}=\beta_{S_i}$ and $\widetilde\alpha_0\#\widetilde\alpha_1\#\cdots\#\widetilde\alpha_{k-1}=\beta$ by definition. Then since $\beta_{S_i}<\omega^{\alpha_i}$ for all $i\in k$ and $\beta=\alpha_0\odot\alpha_1\odot\cdots\odot\alpha_{k-1}$, we have $\widetilde\alpha_i<\alpha_i$ for all $i\in k$ while $\widetilde\alpha_0\#\widetilde\alpha_1\#\cdots\#\widetilde\alpha_{k-1}=\alpha_0\odot\alpha_1\odot\cdots\odot\alpha_{k-1}$, contrary to the definition of the Milner--Rado sum.
\end{proof}

This completes the proof of Theorem \ref{link}, which provides us with the topological pigeonhole numbers for finite sequences of countable ordinals when either each ordinal is a power of $\omega$ or each ordinal is a power of $\omega$ plus 1.

Our next result generalises Theorem \ref{linkpowers} by considering mixtures of such ordinals including at least one power of $\omega$. Using monotonicity, this will provide us with the topological pigeonhole numbers for all finite sequences of countable ordinals in which one of the ordinals is a power of $\omega$, thereby completing case \ref{finiteofcountablepower} of the principle. The result essentially says that in this case, the topological pigeonhole number is the same as if the other ordinals were ``rounded up'' to the next largest power of $\omega$.

The proof involves proving two negative relations, the first of which uses ideas from Theorem \ref{linksuccessor} and the second of which uses ideas from Theorem \ref{linkpowers}.

\begin{theorem}\label{powersroundup}
Let $\alpha_0,\alpha_1,\dots,\alpha_l,\delta_{l+1},\delta_{l+2},\dots,\delta_{k-1}\in\omega_1\setminus\left\{0\right\}$, where $l\in k$. Then
\begin{align*}
&P^{top}\left(\omega^{\alpha_0},\omega^{\alpha_1},\dots,\omega^{\alpha_l},\omega^{\delta_{l+1}}+1,\omega^{\delta_{l+2}}+1,\dots,\omega^{\delta_{k-1}}+1\right)\\
&=\omega^{\alpha_0\odot\alpha_1\odot\cdots\alpha_l\odot\left(\delta_{l+1}+1\right)\odot\left(\delta_{l+2}+1\right)\odot\cdots\odot\left(\delta_{k-1}+1\right)}.
\end{align*}
\end{theorem}

\begin{proof}
Write $P$ for the left-hand side and $\beta$ for the right-hand side. Clearly $P\leq\beta$ by Theorem \ref{linkpowers} and monotonicity, so we prove that $P\geq\beta$.

Suppose first that $\alpha_i$ is a successor ordinal for all $i\in\left\{0,1,\dots,l\right\}$, say $\alpha_i=\delta_i+1$. Then by Theorem \ref{milnerrado}, $\beta=\omega^{\delta_0\#\delta_1\#\cdots\#\delta_{k-1}+1}$. Suppose $\zeta<\beta$. We will show that
\[\zeta\nrightarrow\left(top\,\omega^{\delta_0+1},\omega^{\delta_1}+1,\omega^{\delta_2}+1,\dots,\omega^{\delta_{k-1}}+1\right)^1,\]
which suffices. Write $\delta=\delta_0\#\delta_1\#\cdots\#\delta_{k-1}$, and observe first that $\zeta\leq\omega^\delta\cdot m+1$ for some $m\in\omega$, so it is sufficient to consider the case in which $\zeta=\omega^\delta\cdot m+1$. Next recall from the proof of Theorem \ref{linksuccessor} that there is a colouring $d:\omega^\delta\rightarrow k$ with the property that $d^{-1}\left(\left\{i\right\}\right)^{\left(\delta_i\right)}=\emptyset$ for all $i\in k$. Now define a colouring $c:\zeta\rightarrow k$ by
\[
c\left(x\right)=
\begin{cases}
d\left(y\right),&\text{if $x=\omega^\delta\cdot l+y$ with $l\in\omega$ and $0<y<\omega^\delta$}\\
0,&\text{if $x\in\left\{0,\omega^\delta,\omega^\delta\cdot 2,\dots,\omega^\delta\cdot m\right\}$}.
\end{cases}
\]
Then for all $i\in\left\{1,2,\dots,k-1\right\}$, $c^{-1}\left(\left\{i\right\}\right)^{\left(\delta_i\right)}=\emptyset$, whereas $\left(\omega^{\delta_i}+1\right)^{\left(\delta_i\right)}=\left\{\omega^{\delta_i}\right\}$, so $c^{-1}\left(\left\{i\right\}\right)$ cannot contain a homeomorphic copy of $\omega^{\delta_i}+1$. On the other hand, $c^{-1}\left(\left\{0\right\}\right)^{\left(\delta_0\right)}\subseteq\left\{\omega^\delta,\omega^\delta\cdot 2,\dots,\omega^\delta\cdot m\right\}$, whereas $\left(\omega^{\delta_0+1}\right)^{\left(\delta_0\right)}$ is infinite, so $c^{-1}\left(\left\{0\right\}\right)$ cannot contain a homeomorphic copy of $\omega^{\delta_0+1}$. This completes the proof for this case.

Suppose instead that $\alpha_j$ is a limit ordinal for some $j\in\left\{0,1,\dots,l\right\}$. Write $\beta=\omega^\delta$. Then by Theorem \ref{milnerrado}, $\delta$ is a limit ordinal. This observation enables us to complete the proof using simpler version of the argument from Theorem \ref{linkpowers}. Suppose $\zeta<\beta$. We will show that
\[\zeta\nrightarrow\left(top\,\omega^{\alpha_0},\omega^{\alpha_1},\dots,\omega^{\alpha_l},\omega^{\delta_{l+1}}+1,\omega^{\delta_{l+2}}+1,\dots,\omega^{\delta_{k-1}}+1\right)^1.\]
Observe first that since $\delta$ is a limit ordinal, $\zeta\leq\omega^\eta$ for some $\eta<\delta$, so it is sufficient to consider the case in which $\zeta=\omega^\eta$. Write $\alpha_i=\delta_i+1$ for all $i\in\left\{l+1,l+2,\dots,k-1\right\}$, and recall from theorem \ref{milnerrado} that $\delta=P^{ord}\left(\alpha_0,\alpha_1,\dots,\alpha_{k-1}\right)$. Since $\eta<\delta$, there is a colouring $\widetilde c:\eta\rightarrow k$ such that for all $i\in k$, the order type of $\widetilde c^{-1}\left(\left\{i\right\}\right)$ is $\widetilde\alpha_i<\alpha_i$. Define a colouring $c:\zeta\rightarrow k$ by $c=\widetilde c\circ\operatorname{CB}$. By Proposition \ref{rankcolouring}, $c^{-1}\left(\left\{i\right\}\right)^{\left(\widetilde\alpha_i\right)}=\emptyset$ for all $i\in k$. However, $\left(\omega^{\alpha_i}\right)^{\left(\widetilde\alpha_i\right)}$ is infinite for all $i\in\left\{0,1,\dots,l\right\}$, and $\left(\omega^{\delta_i}+1\right)^{\left(\widetilde\alpha_i\right)}\supseteq\left\{\omega^{\delta_i}\right\}$ for all $i\in\left\{l+1,l+2,\dots,k-1\right\}$. Hence $c^{-1}\left(\left\{i\right\}\right)$ cannot contain a homeomorphic copy of $\omega^{\alpha_i}$ (if $i\in\left\{0,1,\dots,l\right\}$) or $\omega^{\delta_i}+1$ (if $i\in\left\{l+1,l+2,\dots,k-1\right\}$).
\end{proof}

Next we move beyond powers of $\omega$ and powers of $\omega$ plus 1 to consider ordinals of the form $\omega^\alpha\cdot m+1$ with $\alpha\in\omega_1\setminus\left\{0\right\}$ and $m\in\omega\setminus\left\{0\right\}$. At this point considerations from the finite pigeonhole principle come into play.

At the same time we will also consider finite ordinals, since they behave in a similar fashion: just as $\omega^\alpha\cdot m+1$ is homeomorphic to the topological disjoint union of $m$ copies of $\omega^\alpha+1$, so $m\in\omega$ is homeomorphic to the topological disjoint union of $m$ copies of $1$. In order to consider both forms of ordinal at the same time we therefore make the following definition.

\begin{definition}
Let $\alpha$ be an ordinal and $m\in\omega\setminus\left\{0\right\}$. We define
\[
\overline\omega\left[\alpha,m\right]=
\begin{cases}
\omega^\alpha\cdot m+1,&\text{if $\alpha>0$}\\
m,&\text{if $\alpha=0$}.
\end{cases}
\]
\end{definition}

The following result deals with finite sequences of countable ordinals of the form $\overline\omega\left[\alpha,m\right]$. It generalises both Theorem \ref{linksuccessor} and the finite pigeonhole principle, and the proof essentially combines these two theorems.

\begin{theorem}\label{simplemultiples}
Let $\alpha_0,\alpha_1,\dots,\alpha_{k-1}\in\omega_1$ and $m_1,m_2,\dots,m_{k-1}\in\omega\setminus\left\{0\right\}$. Then
\[P^{top}\left(\overline\omega\left[\alpha_0,m_0\right],\overline\omega\left[\alpha_1,m_1\right],\dots,\overline\omega\left[\alpha_{k-1},m_{k-1}\right]\right)=\overline\omega\left[\alpha,m\right],\]
where $\alpha=\alpha_0\#\alpha_1\#\cdots\#\alpha_{k-1}$ and $m=\sum_{i=0}^{k-1}\left(m_i-1\right)+1$.
\end{theorem}

\begin{proof}
We assume for simplicity that $\alpha_i>0$ for all $i\in k$, the other case being no harder. Thus $\overline\omega\left[\alpha_i,m_i\right]=\omega^{\alpha_i}\cdot m_i+1$ for all $i\in k$ and $\overline\omega\left[\alpha,m\right]=\omega^\alpha\cdot m+1$.

To see that
\[\omega^\alpha\cdot m\nrightarrow\left(top\,\omega^{\alpha_0}\cdot m_0+1,\omega^{\alpha_1}\cdot m_1+1,\dots,\omega^{\alpha_{k-1}}\cdot m_{k-1}+1\right)^1,\]
recall from the proof of Theorem \ref{linksuccessor} that there is a colouring $d:\omega^\alpha\rightarrow k$ with the property that $d^{-1}\left(\left\{i\right\}\right)^{\left(\alpha_i\right)}=\emptyset$ for all $i\in k$. Additionally observe that since $m-1\nrightarrow\left(m_0,m_1,\dots,m_{k-1}\right)^1$, there is a colouring $e:\left\{1,2,\dots,m-1\right\}\rightarrow k$ with the property that $\left|e^{-1}\left(\left\{i\right\}\right)\right|\leq m_i-1$ for all $i\in k$. Now define a colouring $c:\omega^\alpha\cdot m\rightarrow k$ by
\[
c\left(x\right)=
\begin{cases}
d\left(y\right),&\text{if $x=\omega^\alpha\cdot l+y$ with $l\in\omega$ and $0<y<\omega^\alpha$}\\
e\left(l\right),&\text{if $x=\omega^\alpha\cdot l$ with $l\in\left\{1,2,\dots,m-1\right\}$}\\
0,&\text{if $x=0$}.
\end{cases}
\]
Then for all $i\in k$,
\[\left|c^{-1}\left(\left\{i\right\}\right)^{\left(\alpha_i\right)}\right|\leq m_i-1\]
by construction, whereas
\[\left|\left(\omega^{\alpha_i}\cdot m_i+1\right)^{\left(\alpha_i\right)}\right|=m_i.\]
Hence $c^{-1}\left(\left\{i\right\}\right)$ cannot contain a homeomorphic copy of $\omega^{\alpha_i}\cdot m_i+1$.

To see that
\[\omega^\alpha\cdot m+1\rightarrow\left(top\,\omega^{\alpha_0}\cdot m_0+1,\omega^{\alpha_1}\cdot m_1+1,\dots,\omega^{\alpha_{k-1}}\cdot m_{k-1}+1\right)^1,\]
let $c:\omega^\alpha\cdot m+1\rightarrow k$. Observe that for each $j\in m$, $\left[\omega^\alpha\cdot j+1,\omega^\alpha\cdot\left(j+1\right)\right]\cong\omega^\alpha+1$. Therefore by Theorem \ref{linksuccessor} there exists $i_j\in k$ and $X_j\subseteq c^{-1}\left(\left\{i_j\right\}\right)\cap\left[\omega^\alpha\cdot j+1,\omega^\alpha\cdot\left(j+1\right)\right]$ with $X_j\cong\omega^{\alpha_{i_j}}+1$. Next observe that by the finite pigeonhole principle, $m\rightarrow\left(m_0,m_1,\dots,m_{k-1}\right)^1$. Hence there exists $i\in k$ such that $\left|\left\{j\in m:i_j=i\right\}\right|\geq m_i$, say $S\subseteq\left\{j\in m:i_j=i\right\}$ satisfies $\left|S\right|=m_i$. But then $\bigcup_{j\in S}X_j$ is a homeomorphic copy of $\omega^{\alpha_i}\cdot m_i+1$ in colour $i$.
\end{proof}

We conclude this section by considering at last ordinals of the form $\omega^\alpha\cdot\left(m+1\right)$ with $\alpha\in\omega_1\setminus\left\{0\right\}$ and $m\in\omega\setminus\left\{0\right\}$. Such an ordinal is homeomorphic to the topological disjoint union of $\omega^\alpha\cdot m+1$ and $\omega^\alpha$ and behaves similarly to $\omega^\alpha\cdot m+1$, but there are additional complications.

The following simple consequence of Theorem \ref{linkpowers} will be useful for finding extra homeomorphic copies of $\omega^\alpha$ in the required colour.

\begin{lemma}\label{findinspare}
Let $\alpha_0,\alpha_1,\dots,\alpha_{k-1}\in\omega_1\setminus\left\{0\right\}$, let $\alpha=\alpha_0\#\alpha_1\#\cdots\#\alpha_{k-1}$ and let $c:\omega^\alpha\rightarrow k$ be a colouring. Then either $c^{-1}\left(\left\{j\right\}\right)$ contains a homeomorphic copy of $\omega^{\alpha_j+1}$ for some $j\in k$, or $c^{-1}\left(\left\{i\right\}\right)$ contains a homeomorphic copy of $\omega^{\alpha_i}$ for all $i\in k$.
\end{lemma}

\begin{proof}
Fix $i\in k$. It is sufficient to prove that either $c^{-1}\left(\left\{i\right\}\right)$ contains a homeomorphic copy of $\omega^{\alpha_i}$, or $c^{-1}\left(\left\{j\right\}\right)$ contains a homeomorphic copy of $\omega^{\alpha_j+1}$ for some $j\in k\setminus\left\{i\right\}$. To see this, simply observe that by Theorem \ref{milnerrado} (or by inspection),
\[\left(\alpha_0+1\right)\odot\cdots\odot\left(\alpha_{i-1}+1\right)\odot\alpha_i\odot\left(\alpha_{i+1}+1\right)\odot\cdots\odot\left(\alpha_{k-1}+1\right)\leq\alpha\]
and hence by Theorem \ref{linkpowers},
\[\omega^\alpha\rightarrow\left(top\,\omega^{\alpha_0+1},\dots,\omega^{\alpha_{i-1}+1},\omega^{\alpha_i},\omega^{\alpha_{i+1}+1},\dots,\omega^{\alpha_{k-1}+1}\right)^1.\qedhere\]
\end{proof}

In our next result we use this to narrow the topological pigeonhole number down to one of two possibilities.

\begin{theorem}\label{multiplemixtures}
Let $\alpha_0,\alpha_1,\dots,\alpha_l\in\omega_1\setminus\left\{0\right\}$, $\alpha_{l+1},\alpha_{l+2},\dots,\alpha_{k-1}\in\omega_1$ and $m_0,m_1,\dots,m_{k-1}\in\omega\setminus\left\{0\right\}$, where $l\in k$. Then
\[P^{top}\left(\omega^{\alpha_0}\cdot\left(m_0+1\right),\dots,\omega^{\alpha_l}\cdot\left(m_l+1\right),\overline\omega\left[\alpha_{l+1},m_{l+1}\right],\dots,\overline\omega\left[\alpha_{k-1},m_{k-1}\right]\right)\]
is equal to either $\omega^\alpha\cdot m+1$ or $\omega^\alpha\cdot\left(m+1\right)$, where $\alpha=\alpha_0\#\alpha_1\#\cdots\#\alpha_{k-1}$ and $m=\sum_{i=0}^{k-1}\left(m_i-1\right)+1$.
\end{theorem}

\begin{proof}
Write $P$ for the topological pigeonhole number in the statement of the theorem. Recall that by Proposition \ref{pseudosummary}, it is sufficient to prove that $\omega^\alpha\cdot m+1\leq P\leq\omega^\alpha\cdot\left(m+1\right)$. The first inequality follows immediately from Theorem \ref{simplemultiples} and monotonicity since $\omega^{\alpha_i}\cdot\left(m_0+1\right)>\overline\omega\left[\alpha_i,m_i\right]$ for all $i\in\left\{0,1,\dots,l\right\}$. The second inequality states that
\[\omega^\alpha\cdot\left(m+1\right)\rightarrow\left(top\,\omega^{\alpha_0}\cdot\left(m_0+1\right),\dots,\omega^{\alpha_l}\cdot\left(m_l+1\right),\overline\omega\left[\alpha_{l+1},m_{l+1}\right],\dots,\overline\omega\left[\alpha_{k-1},m_{k-1}\right]\right)^1.\]

To see this, let $c:\omega^\alpha\cdot\left(m+1\right)\rightarrow k$. First note that for $i\in\left\{0,1,\dots,l\right\}$, $\omega^{\alpha_i}\cdot\left(m_i+1\right)$ is homeomorphic to the topological disjoint union of $\overline\omega\left[\alpha_i,m_i\right]=\omega^{\alpha_i}\cdot m_i+1$ and $\omega^{\alpha_i}$. Now by Theorem \ref{simplemultiples}, there exists $i\in k$ and $X\subseteq c^{-1}\left(\left\{i\right\}\right)\cap\left(\omega^\alpha\cdot m+1\right)$ with $X\cong\overline\omega\left[\alpha_i,m_i\right]$. If $i\in\left\{l+1,l+2,\dots,k-1\right\}$, then we are done, so assume $i\in\left\{0,1,\dots,l\right\}$. Next consider the restriction of $c$ to $\left[\omega^\alpha\cdot m+1,\omega^\alpha\cdot\left(m+1\right)\right)$, which is homeomorphic to $\omega^\alpha$. By Lemma \ref{findinspare}, either $c^{-1}\left(\left\{j\right\}\right)\cap\left[\omega^\alpha\cdot m+1,\omega^\alpha\cdot\left(m+1\right)\right)$ contains a homeomorphic copy of $\omega^{\alpha_j+1}$ for some $j\in k$, in which case we are done, or there exists $Y\subseteq c^{-1}\left(\left\{i\right\}\right)\cap\left[\omega^\alpha\cdot m+1,\omega^\alpha\cdot\left(m+1\right)\right)$ with $Y\cong\omega^{\alpha_i}$, in which case $X\cup Y$ is a homeomorphic copy of $\omega^{\alpha_i}\cdot\left(m_i+1\right)$ in colour $i$.
\end{proof}

Recall that by Lemmas \ref{pseudorelevant} and \ref{pseudopositive} it is enough to consider ordinals of the form $\omega^\alpha\cdot m$ and $\omega^\alpha\cdot m+1$ with $m\in\omega\setminus\left\{0\right\}$. It follows that Theorems \ref{simplemultiples} and \ref{multiplemixtures} together cover case \ref{finiteofcountableother} of the principle. Thus to complete case \ref{finiteofcountable} it remains only to distinguish between the two possibilities presented by Theorem \ref{multiplemixtures}.

In our final result of this section we do this for the case in which $m_i=1$ for all $i\in k\setminus\left\{0\right\}$. In particular this completes case \ref{distinguishingcase}. At this point the Cantor--Bendixson ranks of ordinal \emph{exponents} come into play. They essentially determine whether or not the negative relation can be proved using the type of colouring given in the first half of the proof of Theorem \ref{linksuccessor}.

\begin{theorem}\label{distinguishing}
Let $\alpha_0,\alpha_1,\dots,\alpha_{k-1}\in\omega_1\setminus\left\{0\right\}$, $m_0\in\omega\setminus\left\{0\right\}$ and $l\in k$. Assume without loss of generality that if $m_0=1$ then $\operatorname{CB}\left(\alpha_0\right)\leq\operatorname{CB}\left(\alpha_i\right)$ for all $i\in\left\{1,2,\dots,l\right\}$. Then
\[\omega^\alpha\cdot m_0+1\rightarrow\left(top\,\omega^{\alpha_0}\cdot\left(m_0+1\right),\omega^{\alpha_1}\cdot 2,\dots,\omega^{\alpha_l}\cdot 2,\omega^{\alpha_{l+1}}+1,\dots,\omega^{\alpha_{k-1}}+1\right)^1\]
if and only if $\operatorname{CB}\left(\alpha_h\right)<\operatorname{CB}\left(\alpha_0\right)$ for some $h\in k$, where $\alpha=\alpha_0\#\alpha_1\#\cdots\#\alpha_{k-1}$.
\end{theorem}

We will prove the ``if'' part by combining Lemma \ref{findinspare} with the following result.

\begin{lemma}\label{distinguishinglemma}
Let $\alpha_0,\alpha_1,\dots,\alpha_{k-1}\in\omega_1\setminus\left\{0\right\}$, let $\alpha=\alpha_0\#\alpha_1\#\cdots\#\alpha_{k-1}$ and let $c:\omega^\alpha+1\rightarrow k$ be a colouring. Then there exists $j\in k$ such that either $c^{-1}\left(\left\{j\right\}\right)$ contains a homeomorphic copy of $\omega^{\alpha_j}\cdot 2$, or $c^{-1}\left(\left\{j\right\}\right)$ contains a homeomorphic copy of $\omega^{\alpha_j}+1$ and $\operatorname{CB}\left(\alpha_j\right)\leq\operatorname{CB}\left(\alpha_i\right)$ for all $i\in k$.
\end{lemma}

The proof of this lemma uses ideas from the proof of Weiss's theorem \cite[Theorem 2.3]{baumgartner}. In particular we will make use of the following result, which was also published by Baumgartner \cite[Lemma 2.6]{baumgartner}.

\begin{lemma}[Weiss]
Let $\alpha\in\omega_1$ not be a power of $\omega$. Write $\alpha=\omega^{\gamma_1}+\omega^{\gamma_2}+\cdots+\omega^{\gamma_n}$ with $\gamma_1\geq\gamma_2\geq\dots\geq\gamma_n$ and $n>1$, and let $\delta=\omega^{\omega^{\gamma_1}+\omega^{\gamma_2}+\cdots+\omega^{\gamma_{n-1}}}$ and $\varepsilon=\omega^{\omega^{\gamma_n}}$. Suppose $Z\subseteq\left\{\delta\cdot x: x\in\varepsilon\right\}$ is order-homeomorphic to $\varepsilon$, say $Z=\left\{z_\eta:\eta\in\varepsilon\right\}$. Then for each $\eta\in\varepsilon$ there exists $Y_\eta\subseteq\left(z_\eta,z_{\eta+1}\right)$ such that $Y_\eta$ is order-homeomorphic to $\delta$ and $Y_\eta$ is cofinal in $\left(z_\eta,z_{\eta+1}\right)$.
\end{lemma}

\begin{proof}[Proof of Lemma \ref{distinguishinglemma}]
Write $\alpha=\omega^{\gamma_1}+\omega^{\gamma_2}+\cdots+\omega^{\gamma_n}$ with $\gamma_1\geq\gamma_2\geq\dots\geq\gamma_n$, and observe that for $j\in k$, the condition that $\operatorname{CB}\left(\alpha_j\right)\leq\operatorname{CB}\left(\alpha_i\right)$ for all $i\in k$ is equivalent to the condition that $\operatorname{CB}\left(\alpha_j\right)=\gamma_n$. The case $k=1$ is trivial, so assume $k>1$ (and hence $n>1$) and let $\delta=\omega^{\omega^{\gamma_1}+\omega^{\gamma_2}+\cdots+\omega^{\gamma_{n-1}}}$ and $\varepsilon=\omega^{\omega^{\gamma_n}}$ as in Weiss's lemma.

First let $c\left(\omega^\alpha\right)=j_0$. Next, by Corollary \ref{weisscorollary} there exists $j_1\in k$ and $Z\subseteq c^{-1}\left(\left\{j_1\right\}\right)\cap\left\{\delta\cdot x: x\in\varepsilon\right\}$ such that $Z$ is cofinal in $\left\{\delta\cdot x: x\in\varepsilon\right\}$ (and hence in $\omega^\alpha$) and $Z$ is order-homeomorphic to $\varepsilon$, say $Z=\left\{z_\eta:\eta\in\varepsilon\right\}$. For each $\eta\in\varepsilon$ choose $Y_\eta$ as in Weiss's lemma. Then for each $\eta\in\varepsilon$, by Corollary \ref{weisscorollary} there exists a partition of $\left\{1,2,\dots,n-1\right\}$ into $k$ pieces $S^\eta_0,S^\eta_1,\dots,S^\eta_{k-1}$ and for each $i\in k$ a subset $X^\eta_i\subseteq c^{-1}\left(\left\{i\right\}\right)\cap Y_\eta$ such that for all $i\in k$, $X^\eta_i$ is order-homeomorphic to $\delta_{S^\eta_i}$ and $X^\eta_i$ is either empty or cofinal in $Y_\eta$ (and hence in $\left(z_\eta,z_{\eta+1}\right)$). Moreover since $\varepsilon\rightarrow\left(\varepsilon\right)^1_r$ for all $r\in\omega\setminus\left\{0\right\}$ (either using Theorem \ref{linkpowers} or simply from the fact that $\varepsilon$ is a power of $\omega$), there exists $T\subseteq\varepsilon$ of order type $\varepsilon$ and a single partition of $\left\{1,2,\dots,n-1\right\}$ into $k$ pieces $S_0,S_1,\dots,S_{k-1}$ such that $S^\eta_i=S_i$ for all $\eta\in T$ and all $i\in k$.

Now if $\delta_{S_j}>\omega^{\alpha_j}$ for some $j\in k$, then we are done. So we may assume $\delta_{S_i}\leq\omega^{\alpha_i}$ for all $i\in k$. But then there must exist $j_2\in k$ with $\operatorname{CB}\left(\alpha_{j_2}\right)=\gamma_n$ such that in fact $\delta_{S_i}=\omega^{\alpha_i}$ for all $i\in k\setminus\left\{j_2\right\}$ and $\delta_{S_{j_2}\cup\left\{n\right\}}=\omega^{\alpha_{j_2}}$.

There are now three possibilities.
\begin{itemize}
\item
If $j_1\neq j_2$, then take $j=j_1$. Pick $\eta_1,\eta_2\in T$ and take
\[X=X^{\eta_1}_j\cup\left\{z_{\eta_1+1}\right\}\cup X^{\eta_2}_j.\]
Then $X$ is a homeomorphic copy of $\omega^{\alpha_j}\cdot 2$ in colour $j$.
\item
If $j_0\neq j_2$, then take $j=j_0$. We now use an argument from the proof of Weiss's theorem. Let $\left(\eta_r\right)_{r\in\omega}$ be a strictly increasing cofinal sequence from $T$ and let $\left(\zeta_r\right)_{r\in\omega\setminus\left\{0\right\}}$ be a strictly increasing cofinal sequence from $\omega^{\alpha_j}$, so that $\omega^{\alpha_j}$ is homeomorphic to the topological disjoint union of the collection $\left(\zeta_r+1\right)_{r\in\omega\setminus\left\{0\right\}}$. For each $r\in\omega\setminus\left\{0\right\}$, choose $W_r\subseteq X^{\eta_r}_j$ with $W_r\cong\zeta_r+1$, and take
\[X=X^{\eta_0}_j\cup\bigcup_{r\in\omega\setminus\left\{0\right\}}W_r\cup\left\{\omega^\alpha\right\}.\]
Then $X$ is a homeomorphic copy of $\omega^{\alpha_j}\cdot 2$ in colour $j$.
\item
If $j_0=j_1=j_2$, then take $j$ to be their common value. We now use another argument from the proof of Weiss's theorem. Let $Z_1$ be the closure of $\left\{z_{\eta+1}:\eta\in T\right\}$ in $Z$ and take
\[X=\bigcup_{\eta\in T}X^\eta_j\cup Z_1\cup\left\{\omega^\alpha\right\}.\]
Then $X$ is a homeomorphic copy of $\omega^{\alpha_j}+1$ in colour $j$, and since $j=j_2$ we have $\operatorname{CB}\left(\alpha_j\right)=\gamma_n$.\qedhere
\end{itemize}
\end{proof}

\begin{proof}[Proof of Theorem \ref{distinguishing}]
Write $\alpha=\omega^{\gamma_1}+\omega^{\gamma_2}+\cdots+\omega^{\gamma_n}$ with $\gamma_1\geq\gamma_2\geq\dots\geq\gamma_n$, and let $\beta=\omega^\alpha\cdot m_0+1$.

Suppose first that $\operatorname{CB}\left(\alpha_0\right)\leq\operatorname{CB}\left(\alpha_i\right)$ for all $i\in k$. As in the proof of Theorem \ref{linksuccessor}, observe that by definition of the natural sum, there is a partition of $\left\{1,2,\dots,n\right\}$ into $k$ pieces $S_0,S_1,\dots,S_{k-1}$ such that for all $i\in k$, $\alpha_i=\sum_{j\in S_i}\omega^{\gamma_j}$. Moreover, since $\operatorname{CB}\left(\alpha_0\right)\leq\operatorname{CB}\left(\alpha_i\right)$ for all $i\in k$ we may assume that $n\in S_0$. Now define a colouring $c:\beta\rightarrow k$ as follows. If $\operatorname{CB}\left(x\right)<\alpha$ (i.e., $x\notin\left\{\omega^\alpha,\omega^\alpha\cdot 2,\dots,\omega^\alpha\cdot m_0\right\}$), then as in Theorem \ref{linksuccessor}, for each $i\in k$ set $c\left(x\right)=i$ if and only if
\[\omega^{\gamma_1}+\omega^{\gamma_2}+\cdots+\omega^{\gamma_{j-1}}\leq\operatorname{CB}\left(x\right)<\omega^{\gamma_1}+\omega^{\gamma_2}+\cdots+\omega^{\gamma_j}\]
for some $j\in S_i$. If $\operatorname{CB}\left(x\right)=\alpha$, then set $c\left(x\right)=0$. If $i\in k\setminus\left\{0\right\}$, then as in Theorem \ref{linksuccessor} $c^{-1}\left(\left\{i\right\}\right)$ cannot contain a homeomorphic copy of $\omega^{\alpha_i}+1$. To deal with the case $i=0$, let $\eta=\sum_{j\in S_0\setminus\left\{n\right\}}\omega^{\gamma_j}$. By the proof of Proposition \ref{rankcolouring},
\[c^{-1}\left(\left\{0\right\}\right)^{\left(\eta\right)}=\left\{x\in\beta:\operatorname{CB}\left(x\right)\geq\omega^{\gamma_1}+\omega^{\gamma_2}+\cdots+\omega^{\gamma_{n-1}}\right\}\cong\omega^{\omega^{\gamma_n}}\cdot m_0+1,\]
whereas $\left(\omega^{\alpha_0}\cdot\left(m_0+1\right)\right)^{\left(\eta\right)}\cong\omega^{\omega^{\gamma_n}}\cdot\left(m_0+1\right)$. It follows by part \ref{pseudodifficultcase} of Proposition \ref{pseudonegative} that $c^{-1}\left(\left\{0\right\}\right)$ cannot contain a homeomorphic copy of $\omega^{\alpha_0}\cdot\left(m_0+1\right)$.

Suppose instead that $\operatorname{CB}\left(\alpha_h\right)<\operatorname{CB}\left(\alpha_0\right)$ for some $h\in k$. If $m_0=1$, then by assumption $\operatorname{CB}\left(\alpha_0\right)\leq\operatorname{CB}\left(\alpha_i\right)$ for all $i\in\left\{1,2,\dots,l\right\}$, so $\operatorname{CB}\left(\alpha_h\right)<\operatorname{CB}\left(\alpha_i\right)$ for all $i\in\left\{1,2,\dots,l\right\}$ and we are done by Lemma \ref{distinguishinglemma}. So assume $m_0>1$. Then for each $p\in m_0$ apply Lemma \ref{distinguishinglemma} to obtain $j_p\in k$ and $X_p\subseteq c^{-1}\left(\left\{j_p\right\}\right)\cap\left[\omega^\alpha\cdot p+1,\omega^\alpha\cdot\left(p+1\right)\right]$ such that either $X_p\cong\omega^{\alpha_{j_p}}\cdot 2$, or $X_p\cong\omega^{\alpha_{j_p}}+1$ and $\operatorname{CB}\left(\alpha_{j_p}\right)\leq\operatorname{CB}\left(\alpha_i\right)$ for all $i\in k$. If $j_p=0$ for all $p\in m_0$, then $X_p\cong\omega^{\alpha_0}\cdot 2$ for all $p\in m_0$ and so $\bigcup_{p=0}^{m_0-1}X_p$ contains a homeomorphic copy of $\omega^{\alpha_0}\cdot\left(m_0+1\right)$, and we are done. So assume $j_q\neq 0$ for some $q\in m_0$. Now pick any $r\in m_0\setminus\left\{q\right\}$ and apply Lemma \ref{findinspare} to $\left[\omega^\alpha\cdot r+1,\omega^\alpha\cdot\left(r+1\right)\right)$. Since we would be done if $c^{-1}\left(\left\{j\right\}\right)$ contained a homeomorphic copy of $\omega^{\alpha_j+1}$ for some $j\in k$, we may assume that there exists $Y\subseteq c^{-1}\left(\left\{j_q\right\}\right)\cap\left[\omega^\alpha\cdot r+1,\omega^\alpha\cdot\left(r+1\right)\right)$ with $Y\cong\omega^{\alpha_{j_q}}$. Then $X_q\cup Y$ is a homeomorphic copy of $\omega^{\alpha_{j_q}}\cdot 2$ in colour $j_q$, which suffices.
\end{proof}

We leave the final few considerations pertaining to case \ref{proofquotedcase} for later.

\subsection{Arbitrary sequences of ordinals at most \texorpdfstring{$\omega_1$}{omega\_1}}\label{justoneomega1andinfiniteofcountable}

We now move on to cases \ref{independencecase}, \ref{justoneomega1} and \ref{infiniteofcountablecase} of the principle, in which no ordinal exceeds $\omega_1$ but either there are infinitely many ordinals or there is at least one ordinal equal to $\omega_1$. Here the arguments are less combinatorial and more set-theoretical than in the previous section, and stationary sets are ubiquitous.

We will cover cases \ref{justoneomega1} and \ref{infiniteofcountablecase} in this section and leave the independence results of case \ref{independencecase} to the next section.

To understand the relevance of club sets, recall Proposition \ref{ordertypehomeomorphic}. From that result it follows that given $X\subseteq\omega_1$, if $X$ is club then $X\cong\omega_1$, and in fact the converse also holds in this case.

The essential reason for the ubiquity of stationary sets in this section is the following result of Friedman \cite{friedman}.

\begin{theorem}[Friedman]\label{friedmanbasic}
Let $S\subseteq\omega_1$ be a stationary set, and let $\alpha\in\omega_1$. Then $S$ has a subset order-homeomorphic to $\alpha$.
\end{theorem}

We will need a slightly more general version of this result. In order to state it we make the following definition.

\begin{definition}
Let $\lambda$ be an uncountable regular cardinal. Define
\[E^\lambda_\omega=\left\{x\in\lambda:\operatorname{cf}\left(x\right)=\omega\right\}.\]
\end{definition}

Note that $E^\lambda_\omega$ is stationary in $\lambda$.

Here is our generalisation of Friedman's theorem.

\begin{theorem}\label{friedmangeneralisation}
Let $\lambda$ be an uncountable regular cardinal, let $S\subseteq E^\lambda_\omega$ be stationary in $\lambda$, and let $\alpha\in\omega_1$. Then $S$ has a subset order-homeomorphic to $\alpha$.
\end{theorem}

\begin{proof}
The proof is essentially identical to the proof of Friedman's theorem \cite{friedman}.
\end{proof}

Our final introductory result is a well-known property of stationary sets.

\begin{lemma}\label{stationaryunions}
Let $\lambda$ be an uncountable regular cardinal, let $S\subseteq\lambda$ be stationary, and let $c:S\rightarrow\kappa$ for some cardinal $\kappa<\lambda$. Then $c^{-1}\left(\left\{i\right\}\right)$ is stationary in $\lambda$ for some $i\in\kappa$.
\end{lemma}

\begin{proof}
This follows easily from the fact that if $C_i\subseteq\lambda$ is club for all $i\in\kappa$ then $\bigcap_{i\in\kappa}C_i$ is also club.
\end{proof}

We are now ready to deal with cases \ref{justoneomega1} and \ref{infiniteofcountablecase} of the principle. The result for case \ref{infiniteofcountablecase} is an easy consequence of Theorem \ref{friedmangeneralisation} and Lemma \ref{stationaryunions}.

\begin{theorem}\label{infiniteofcountable}
Let $\kappa\geq\aleph_0$ be a cardinal, and let $\alpha_i$ be an ordinal with $2\leq\alpha_i<\omega_1$ for all $i\in\kappa$. Then
\[P^{top}\left(\alpha_i\right)_{i\in\kappa}=\kappa^+.\]
\end{theorem}

\begin{proof}
Clearly if $\zeta<\kappa^+$ then $\zeta\nrightarrow\left(top\,\alpha_i\right)^1_{i\in\kappa}$ by considering an injection $\zeta\rightarrow\kappa$.

To see that $\kappa^+\rightarrow\left(top\,\alpha_i\right)^1_{i\in\kappa}$, let $c:\kappa^+\rightarrow\kappa$. Then by Lemma \ref{stationaryunions} there exists $i\in\kappa$ such that $c^{-1}\left(\left\{i\right\}\right)\cap E^{\kappa^+}_\omega$ is stationary in $\kappa^+$, which by Theorem \ref{friedmangeneralisation} contains a homeomorphic copy of $\alpha_i$.
\end{proof}

The proof for case \ref{justoneomega1} is a little trickier.

\begin{theorem}\label{omega1equal}
Let $\kappa$ be a cardinal, let $\alpha_r=\omega_1$ for some $r\in\kappa$, and let $\alpha_i$ be an ordinal with $2\leq\alpha_i<\omega_1$ for all $i\in\kappa\setminus\left\{r\right\}$. Then
\[P^{top}\left(\alpha_i\right)_{i\in\kappa}=\operatorname{max}\left\{\omega_1,\kappa^+\right\}.\]
\end{theorem}

\begin{proof}
As in the proof of Theorem \ref{infiniteofcountable}, if $\zeta<\kappa^+$ then $\zeta\nrightarrow\left(top\,\alpha_i\right)^1_{i\in\kappa}$. Additionally, if $\zeta<\omega_1$ then $\zeta\nrightarrow\left(top\,\alpha_i\right)^1_{i\in\kappa}$ by considering the constant colouring with colour $r$.

To see that $\operatorname{max}\left\{\omega_1,\kappa^+\right\}\rightarrow\left(top\,\alpha_i\right)^1_{i\in\kappa}$, first observe that the case $\kappa<\aleph_0$ follows from the case $\kappa=\aleph_0$. So we may assume $\kappa\geq\aleph_0$, implying that $\operatorname{max}\left\{\omega_1,\kappa^+\right\}=\kappa^+$. So let $c:\kappa^+\rightarrow\kappa$. Then let
\[Z=\left(c^{-1}\left(\left\{r\right\}\right)\cap E^{\kappa^+}_\omega\right)\cup\left(\kappa^+\setminus E^{\kappa^+}_\omega\right).\]

Suppose first that $Z$ has a subset $C$ that is club in $\kappa^+$. Then $C\cap\left\{\omega\cdot x:x\in\kappa^+\right\}$ is also club. Let the initial segment of this set of order type $\omega_1$ be $Y$, and let $X=Y^\prime$. Then $X\cong\omega_1$ by Proposition \ref{ordertypehomeomorphic}, but in addition $X\subseteq E^{\kappa^+}_\omega$ and hence $X\subseteq c^{-1}\left(\left\{r\right\}\right)$ by definition of $Z$.

Suppose instead that $Z$ has no subset that is club in $\kappa^+$. Then $\kappa^+\setminus Z$ is stationary in $\kappa^+$. But by definition of $Z$,
\[\kappa^+\setminus Z=\bigcup_{i\in\kappa\setminus\left\{r\right\}}\left(c^{-1}\left(\left\{i\right\}\right)\cap E^{\kappa^+}_\omega\right).\]
Hence by Lemma \ref{stationaryunions} there exists $i\in\kappa\setminus\left\{r\right\}$ such that $c^{-1}\left(\left\{i\right\}\right)\cap E^{\kappa^+}_\omega$ is stationary in $\kappa^+$, which by Theorem \ref{friedmangeneralisation} contains a homeomorphic copy of $\alpha_i$.
\end{proof}

\subsection{Independence results}\label{independenceresults}

We now move on to case \ref{independencecase} of the principle, in which no ordinal exceeds $\omega_1$ and two or more ordinals are equal to $\omega_1$. To begin with we quote the following result, a proof of which can be found in Weiss's article \cite[Theorem 2.8]{weiss}. This follows easily from the fact that $\omega_1$ may be written as a disjoint union of two stationary sets.

\begin{proposition}\label{omega2provable}
If $\beta\in\omega_2$ then $\beta\nrightarrow\left(top\,\omega_1\right)^1_2$.
\end{proposition}

\begin{corollary}\label{omega2geq}
Let $\kappa$ be a cardinal, and let $\alpha_i$ be an ordinal with $2\leq\alpha_i\leq\omega_1$ for all $i\in\kappa$. Suppose $\alpha_r,\alpha_s=\omega_1$ for some distinct $r,s\in\kappa$. Then
\[P^{top}\left(\alpha_i\right)_{i\in\kappa}\geq\operatorname{max}\left\{\omega_2,\kappa^+\right\}.\]
\end{corollary}

\begin{proof}
Clearly if $\zeta<\kappa^+$ then $\zeta\nrightarrow\left(top\,\alpha_i\right)^1_{i\in\kappa}$, and if $\zeta<\omega_2$ then $\zeta\nrightarrow\left(top\,\omega_1\right)^1_2$ by Proposition \ref{omega2provable} and hence $\zeta\nrightarrow\left(top\,\alpha_i\right)^1_{i\in\kappa}$.
\end{proof}

We shall now see that, modulo a large cardinal assumption, this is the strongest $\mathsf{ZFC}$-provable statement applicable to case \ref{independencecase}. Recall from the statement of the principle that we write ``$P_\kappa=x$'' for the statement, ``$\kappa$ is a cardinal, and for all sequences $\left(\alpha_i\right)_{i\in\kappa}$ of ordinals, if $2\leq\alpha_i\leq\omega_1$ for all $i\in\kappa$ and $\alpha_r,\alpha_s=\omega_1$ for some distinct $r,s\in\kappa$, then $P^{top}\left(\alpha_i\right)_{i\in\kappa}=x$'', and likewise for ``$P_\kappa\geq x$''.

In one direction, we use the following result of Prikry and Solovay \cite{prikrysolovay}.

\begin{theorem}[Prikry--Solovay]
Suppose $V=L$ and let $\beta$ be any ordinal. Then
\[\beta\nrightarrow\left(top\,\omega_1\right)^1_2.\]
\end{theorem}

\begin{corollary}\label{independenceequality}
If $\mathsf{ZFC}$ is consistent, then so is
\[\mathsf{ZFC}+\text{``for all cardinals $\kappa\geq 2$, $P_\kappa=\infty$''}.\pushQED{\qed}\qedhere\popQED\]
\end{corollary}

\begin{proof}
This follows immediately from the Prikry--Solovay theorem and monotonicity of pigeonhole numbers.
\end{proof}

In the other direction, we use a result of Shelah, who introduced the following notation \cite[Chapter X, \textsection 7]{shelah}.

\begin{definition}[Shelah]
Let $\lambda$ be an uncountable regular cardinal. Write $\operatorname{Fr}^+\left(\lambda\right)$ to mean that every subset of $E^\lambda_\omega$ that is stationary in $\lambda$ has a subset order-homeomorphic to $\omega_1$.
\end{definition}

Note the similarity between this notion and Theorem \ref{friedmangeneralisation}. In fact the letters ``$\operatorname{Fr}$'' here refer to Friedman, who first asked whether or not there exists an ordinal $\beta$ with $\beta\rightarrow\left(top\,\omega_1\right)^1_2$ \cite{friedman}.

Here is the result of Shelah \cite[Chapter XI, Theorem 7.6]{shelah}.

\begin{theorem}[Shelah]\label{shelahsupercompact}
If $\mathsf{ZFC}+$``there exists a supercompact cardinal'' is consistent, then so is 
\[\mathsf{ZFC}+\text{``$\operatorname{Fr}^+\left(\lambda\right)$ holds for every regular cardinal $\lambda\geq\aleph_2$''}.\]
\end{theorem}

In order to apply Shelah's result to case \ref{independencecase} we make the following observation.

\begin{lemma}\label{friedmanplus}
Let $\kappa\geq\aleph_1$ be a cardinal. If $\operatorname{Fr}^+\left(\kappa^+\right)$ holds, then
\[\kappa^+\rightarrow\left(top\,\omega_1\right)^1_\kappa.\]
\end{lemma}

\begin{proof}
Simply apply Lemma \ref{stationaryunions}.
\end{proof}

\begin{corollary}\label{independencesupercompact}
If $\mathsf{ZFC}+$``there exists a supercompact cardinal'' is consistent, then so is
\[\mathsf{ZFC}+\text{``for all cardinals $\kappa\geq 2$, $P_\kappa=\operatorname{max}\left\{\omega_2,\kappa^+\right\}$''}.\]
\end{corollary}

\begin{proof}
Observe that by Corollary \ref{omega2geq}, the following is a theorem of $\mathsf{ZFC}$: ``for all cardinals $\kappa\geq 2$, $P_\kappa\geq\operatorname{max}\left\{\omega_2,\kappa^+\right\}$''. To finish, simply combine Theorem \ref{shelahsupercompact} with Lemma \ref{friedmanplus}.
\end{proof}

To conclude this section, we address the question of whether a large cardinal assumption is required. To this end we give an equiconsistency result essentially due to Silver and Shelah.

Silver proved the following result by showing that if $\omega_2\rightarrow\left(top\,\omega_1\right)^1_2$ then $\square_{\omega_1}$ does not hold, a proof of which can be found in Weiss's article \cite[Theorem 2.10]{weiss}.

\begin{theorem}[Silver]\label{silver}
If $\omega_2\rightarrow\left(top\,\omega_1\right)^1_2$ then $\omega_2$ is Mahlo in $L$.
\end{theorem}

Here is the result of Shelah \cite[Chapter XI, Theorem 7.1]{shelah}.

\begin{theorem}[Shelah]\label{shelahmahlo}
If $\mathsf{ZFC}+$``there exists a Mahlo cardinal'' is consistent, then so is $\mathsf{ZFC}+\text{``$\operatorname{Fr}^+\left(\aleph_2\right)$''}$.
\end{theorem}

\begin{corollary}\label{independencemahlo}
$\mathsf{ZFC}+$``there exists a Mahlo cardinal'' is consistent if and only if
\[\mathsf{ZFC}+\text{``$\omega_2\rightarrow\left(top\,\omega_1\right)^1_2$''}\]
is consistent.
\end{corollary}

\begin{proof}
Theorem \ref{silver} gives the ``if'' statement. The ``only if'' statement follows by combining Theorem \ref{shelahmahlo} with Lemma \ref{friedmanplus}.
\end{proof}

\subsection{Sequences including an ordinal larger than \texorpdfstring{$\omega_1$}{omega\_1}}

It remains to cover cases \ref{infinityprovablecase} and \ref{lopsidedcase} of the principle, in which one of the ordinals exceeds $\omega_1$. Although this appears to be a very large class of cases, the situation is dramatically simplified by the following elementary argument covering case \ref{infinityprovablecase}. It is our only result in which the topological pigeonhole number ($\mathsf{ZFC}$-provably) does not exist.

\begin{proposition}\label{infinityprovable}
$P^{top}\left(\omega_1+1,\omega+1\right)=\infty$.
\end{proposition}

\begin{proof}
Let $\beta$ be any ordinal. We show that $\beta\nrightarrow\left(top\,\omega_1+1,\omega+1\right)^1$. First observe that a homeomorphic copy of $\omega_1+1$ must contain a point of cofinality $\omega_1$, while a homeomorphic copy of $\omega+1$ must contain a point of cofinality $\omega$. The result is then witnessed by the colouring $c:\beta\rightarrow 2$ given by
\[
c\left(x\right)=
\begin{cases}
1,&\text{if $\operatorname{cf}\left(x\right)\geq\omega_1$}\\
0,&\text{otherwise}.
\end{cases}
\qedhere\]
\end{proof}

We conclude this section by simplifying case \ref{lopsidedcase} using another elementary argument. We leave the rest of the proof for this case for the next section.

\begin{lemma}\label{lopsidedsimplification}
Let $\kappa$ be a cardinal and let $\alpha_i$ be an ordinal for each $i\in\kappa$. Suppose $\alpha_r\geq\omega_1+1$ for some $r\in\kappa$ and $2\leq\alpha_i\leq\omega$ for all $i\in\kappa\setminus\left\{r\right\}$, and let
\[
\lambda=
\begin{cases}
\kappa^+,&\text{if $\kappa\geq\aleph_0$}\\
\aleph_0,&\text{if $\kappa<\aleph_0$ and $\alpha_s=\omega$ for some $s\in\kappa\setminus\left\{r\right\}$}\\
\sum_{i\in\kappa\setminus\left\{r\right\}}\left(\alpha_i-1\right)+1,&\text{if $\kappa<\aleph_0$ and $\alpha_i<\omega$ for all $i\in\kappa\setminus\left\{r\right\}$}.
\end{cases}
\]
Let $\beta$ be any ordinal. Then
\[\beta\rightarrow\left(top\,\alpha_i\right)^1_{i\in\kappa}\]
if and only if for every subset $A\subseteq\beta$ with $\left|A\right|<\lambda$ there exists $X\subseteq\beta\setminus A$ with $X\cong\alpha_r$.
\end{lemma}

\begin{proof}
First suppose that $\beta\rightarrow\left(top\,\alpha_i\right)^1_{i\in\kappa}$ and let $A\subseteq\beta$ with $\left|A\right|<\lambda$. If $\kappa\geq\aleph_0$, then take $f:A\rightarrow\kappa\setminus\left\{r\right\}$ to be any injection; if $\kappa<\aleph_0$ and $\alpha_s=\omega$ for some $s\in\kappa\setminus\left\{r\right\}$, then take $f:A\rightarrow\left\{s\right\}$ to be the constant function; and if $\kappa<\aleph_0$ and $\alpha_i<\omega$ for all $i\in\kappa\setminus\left\{r\right\}$, then take $f:A\rightarrow\kappa\setminus\left\{r\right\}$ to be any function with $\left|f^{-1}\left(\left\{i\right\}\right)\right|\leq\alpha_i-1$ for all $i\in\kappa\setminus\left\{r\right\}$. Now define a colouring $c:\beta\rightarrow\kappa$ by
\[
c\left(x\right)=
\begin{cases}
r,&\text{if $x\notin A$}\\
f\left(x\right),&\text{if $x\in A$}.
\end{cases}
\]
Then by construction $\left|c^{-1}\left(\left\{i\right\}\right)\right|<\alpha_i$ for all $i\in\kappa\setminus\left\{r\right\}$, so since $\beta\rightarrow\left(top\,\alpha_i\right)^1_{i\in\kappa}$ there exists $X\subseteq c^{-1}\left(\left\{r\right\}\right)=\beta\setminus A$ with $X\cong\alpha_r$.

Conversely, suppose that for every subset $A\subseteq\beta$ with $\left|A\right|<\lambda$ there exists $X\subseteq\beta\setminus A$ with $X\cong\alpha_r$. Let $c:\beta\rightarrow\kappa$ be a colouring, and let $A=c^{-1}\left(\kappa\setminus\left\{r\right\}\right)$. If $\left|A\right|<\lambda$ then by assumption there exists $X\subseteq\beta\setminus A=c^{-1}\left(\left\{r\right\}\right)$ with $X\cong\alpha_r$ and we are done, so assume $\left|A\right|\geq\lambda$. If $\kappa\geq\aleph_0$, then $\left|c^{-1}\left(\left\{j\right\}\right)\right|\geq\kappa^+$ for some $j\in\kappa\setminus\left\{r\right\}$; if $\kappa<\aleph_0$ and $\alpha_s=\omega$ for some $s\in\kappa\setminus\left\{r\right\}$, then $\left|c^{-1}\left(\left\{j\right\}\right)\right|\geq\aleph_0$ for some $j\in\kappa\setminus\left\{r\right\}$; and if $\kappa<\aleph_0$ and $\alpha_i<\omega$ for all $i\in\kappa\setminus\left\{r\right\}$, then by the finite pigeonhole principle $\left|c^{-1}\left(\left\{j\right\}\right)\right|\geq\alpha_j$ for some $j\in\kappa\setminus\left\{r\right\}$. In every case $\left|c^{-1}\left(\left\{j\right\}\right)\right|\geq\left|\alpha_j\right|$ and we are done.
\end{proof}

\subsection{Proof of the principle}\label{principleproof}

Having provided the key ingredients, we now complete the proof of the principle.

\begin{proof}[Proof of Theorem \ref{principle}]
We split into the same cases as in the statement of the theorem.
\begin{enumerate}
\item
This follows from Proposition \ref{infinityprovable}.
\item
Let $\lambda$ be as in Lemma \ref{lopsidedsimplification}, and note first of all that if $\alpha_r$ is a power of $\omega$ then $P^{top}\left(\alpha_i\right)_{i\in\kappa}\geq\alpha_r$ by part \ref{pseudolimitcase} of Proposition \ref{pseudonegative}.
\begin{enumerate}
\item
In this case $\lambda=\kappa^+$.
\begin{enumerate}
\item
Write $\omega^\beta\cdot m+1\leq\alpha_r\leq\omega^\beta\cdot\left(m+1\right)$ with $\beta$ an ordinal and $m\in\omega\setminus\left\{0\right\}$, and note that $\alpha_r\cdot\kappa^+=\omega^\beta\cdot\kappa^+$.

Suppose $\zeta<\omega^\beta\cdot\kappa^+$. To see that $\zeta\nrightarrow\left(top\,\alpha_i\right)^1_{i\in\kappa}$, let $A=\zeta\cap\left\{\omega^\beta\cdot\eta:\eta\in\kappa^+\setminus\left\{0\right\}\right\}$, so $\left|A\right|<\kappa^+$. Then $\left(\zeta\setminus A\right)^{\left(\beta\right)}=\emptyset$ whereas $\left|\alpha_r^{\left(\beta\right)}\right|=m$, so $\zeta\setminus A$ cannot contain a homeomorphic copy of $\alpha_r$.

To see that $\omega^\beta\cdot\kappa^+\rightarrow\left(top\,\alpha_i\right)^1_{i\in\kappa}$, let $A\subseteq\omega^\beta\cdot\kappa^+$ with $\left|A\right|<\kappa^+$. Then
\[A\subseteq\left(\bigcup_{\eta\in S}\left[\omega^\beta\cdot\eta+1,\omega^\beta\cdot\left(\eta+1\right)\right]\right)\cup\left\{\omega^{\beta+1}\cdot\eta:\eta\in\kappa^+\right\}\]
for some $S\subseteq\kappa^+$ with $\left|S\right|<\kappa^+$. Let $T\subseteq\kappa^+\setminus S$ with $\left|T\right|=m+1$. Then
\[\bigcup_{\eta\in T}\left[\omega^\beta\cdot\eta+1,\omega^\beta\cdot\left(\eta+1\right)\right]\]
is a homeomorphic copy of $\omega^\beta\cdot\left(m+1\right)+1$ disjoint from $A$, which suffices.
\item
Write $\alpha_r=\omega^\beta$. To see that $\alpha_r\cdot\kappa^+\rightarrow\left(top\,\alpha_i\right)^1_{i\in\kappa}$, simply observe that $\alpha_r\cdot\kappa^+=\left(\alpha_r+1\right)\cdot\kappa^+\rightarrow\left(top\,\alpha_r+1,\left(\alpha_i\right)_{i\in\kappa\setminus\left\{r\right\}}\right)^1$ by the previous case and use monotonicity. It remains to show either that $\alpha_r\rightarrow\left(top\,\alpha_i\right)^1_{i\in\kappa}$, or that if $\zeta<\alpha_r\cdot\kappa^+$ then $\zeta\nrightarrow\left(top\,\alpha_i\right)^1_{i\in\kappa}$.
\begin{enumerate}
\item
To see that $\alpha_r\rightarrow\left(top\,\alpha_i\right)^1_{i\in\kappa}$, simply observe that if $A\subseteq\alpha_r$ with $\left|A\right|<\kappa^+$, then $\operatorname{sup}A<\alpha_r$ since $\operatorname{cf}\left(\alpha_r\right)\geq\kappa^+$, and so $\alpha_r\setminus\left[0,\operatorname{sup}A\right]\cong\alpha_r$ since $\alpha_r$ is a power of $\omega$.
\item
Suppose $\zeta<\alpha_r\cdot\kappa^+$. To see that $\zeta\nrightarrow\left(top\,\alpha_i\right)^1_{i\in\kappa}$, let $B\subseteq\alpha_r$ be club with $\left|B\right|=\operatorname{cf}\left(\alpha_r\right)$, and let
\[A=\zeta\cap\left\{\alpha_r\cdot\eta+x:\eta\in\kappa^+,x\in B\cup\left\{0\right\}\right\}.\]
Then $\left|A\right|<\kappa^+$ since $\operatorname{cf}\left(\alpha_r\right)<\kappa^+$. Suppose for contradiction $X\subseteq\zeta\setminus A$ with $X\cong\alpha_r$. Since $\alpha_r$ is a power of $\omega$, using Theorem \ref{orderreinforcing} and passing to a subspace if necessary, we may assume that $X$ is order-homeomorphic to $\alpha_r$. Let $Y=X\cup\left\{\operatorname{sup}X\right\}\cong\alpha_r+1$. Then $Y^{\left(\beta\right)}=\left\{\operatorname{sup}X\right\}$, so by Lemma \ref{cantorbendixson} $\operatorname{sup}X=\alpha_r\cdot\eta$ for some $\eta\in\kappa^+\setminus\left\{0\right\}$. It follows using Proposition \ref{ordertypehomeomorphic} that $X$ is club in $\alpha_r\cdot\eta$. But then $\operatorname{cf}\left(\alpha_r\cdot\eta\right)=\operatorname{cf}\left(\alpha_r\right)>\aleph_0$ and $A$ is also club in $\alpha_r\cdot\eta$, so $X\cap A\neq\emptyset$, contrary to the definition of $X$.
\item
\begin{itemize}
\item
Suppose $\zeta<\alpha_r\cdot\kappa^+$. To see that $\zeta\nrightarrow\left(top\,\alpha_i\right)^1_{i\in\kappa}$, let
\[A=\zeta\cap\left\{\alpha_r\cdot\eta+\omega^\gamma\cdot x:\eta\in\kappa^+,x\in\omega^{\omega^\delta}\right\},\]
so $\left|A\right|<\kappa^+$ since $\delta<\kappa^+$. Then $\left(\zeta\setminus A\right)^{\left(\gamma\right)}=\emptyset$ whereas $\alpha_r^{\left(\gamma\right)}\cong\omega^{\omega^\delta}$, so $\zeta\setminus A$ cannot contain a homeomorphic copy of $\alpha_r$.
\item
First note that since $\delta>0$, either $\delta$ is a successor ordinal or $\operatorname{cf}\left(\delta\right)=\aleph_0$. To see that $\alpha_r\rightarrow\left(top\,\alpha_i\right)^1_{i\in\kappa}$, let $A\subseteq\alpha_r$ with $\left|A\right|<\kappa^+$. Using the fact that $\delta>\kappa^+$, we now choose a strictly increasing cofinal sequence $\left(\beta_n\right)_{n\in\omega}$ from $\beta$ with $\operatorname{cf}\left(\omega^{\beta_n}\right)=\operatorname{cf}\left(\beta_n\right)=\kappa^+$ for all $n\in\omega$. If $\delta=\varepsilon+1$, then take $\beta_n=\gamma+\omega^\varepsilon\cdot n+\kappa^+$ for all $n\in\omega$. If $\operatorname{cf}\left(\delta\right)=\aleph_0$, then let $\left(\delta_n\right)_{n\in\omega}$ be a strictly increasing cofinal sequence from $\delta$ with $\delta_n>\kappa^+$ for all $n\in\omega$, and take $\beta_n=\gamma+\omega^{\delta_n}+\kappa^+$ for all $n\in\omega$. Then for each $n\in\omega$, let $x_n=\operatorname{max}\left\{\omega^{\beta_n},\operatorname{sup}\left(A\cap\omega^{\beta_{n+1}}\right)\right\}$ and let $X_n=\left(x_n,\omega^{\beta_{n+1}}\right)$. Then $X_n\cong\omega^{\beta_{n+1}}$, so there exists $Y_n\subseteq X_n$ with $Y_n\cong\omega^{\beta_n}+1$. Then $\bigcup_{n\in\omega}Y_n$ is a homeomorphic copy of $\alpha_r$ disjoint from $A$.
\end{itemize}
\end{enumerate}
\end{enumerate}
\item
In this case $\lambda=\aleph_0$.
\begin{enumerate}
\item\label{lopsidedproofquotedcase}
To see that $\alpha_r\rightarrow\left(top\,\alpha_i\right)^1_{i\in\kappa}$, simply observe that if $A\subseteq\alpha_r$ with $\left|A\right|<\aleph_0$ then $\alpha_r\setminus\left[0,\operatorname{max}A\right]\cong\alpha_r$.
\item
Write $\omega^\beta\cdot m+1\leq\alpha_r\leq\omega^\beta\cdot\left(m+1\right)$ with $\beta$ an ordinal and $m\in\omega\setminus\left\{0\right\}$, and note that $\alpha_r\cdot\omega=\omega^{\beta+1}$.

Suppose $\zeta<\omega^{\beta+1}$. To see that $\zeta\nrightarrow\left(top\,\alpha_i\right)^1_{i\in\kappa}$, let $A=\zeta\cap\left\{\omega^\beta\cdot n:n\in\omega\setminus\left\{0\right\}\right\}$, which is finite. Then $\left(\zeta\setminus A\right)^{\left(\beta\right)}=\emptyset$ whereas $\left|\alpha_r^{\left(\beta\right)}\right|=m$, so $\zeta\setminus A$ cannot contain a homeomorphic copy of $\alpha_r$.

To see that $\omega^{\beta+1}\rightarrow\left(top\,\alpha_i\right)^1_{i\in\kappa}$, simply observe that $\omega^{\beta+1}\rightarrow\left(top\,\omega^{\beta+1},\left(\alpha_i\right)_{i\in\kappa\setminus\left\{r\right\}}\right)^1$ by the previous case and use monotonicity.
\end{enumerate}
\item
In this case $\lambda=\sum_{i\in\kappa\setminus\left\{r\right\}}\left(\alpha_i-1\right)+1$.
\begin{enumerate}
\item
The result is trivial if $\kappa=1$, and if $\alpha_r$ is a power of $\omega$ then the argument of case \ref{lopsidedproofquotedcase} suffices.
\item
Suppose $\zeta<\omega^\beta\cdot\left(\lambda-1+m\right)+1$. To see that $\zeta\nrightarrow\left(top\,\alpha_i\right)^1_{i\in\kappa}$, let $A=\zeta\cap\left\{\omega^\beta,\omega^\beta\cdot 2,\dots,\omega^\beta\cdot\left(\lambda-1\right)\right\}$, so $\left|A\right|<\lambda$. Then $\left|\left(\zeta\setminus A\right)^{\left(\beta\right)}\right|\leq m-1$ whereas $\left|\alpha_r^{\left(\beta\right)}\right|=m$, so $\zeta\setminus A$ cannot contain a homeomorphic copy of $\alpha_r$.

To see that $\omega^\beta\cdot\left(\lambda-1+m\right)+1\rightarrow\left(top\,\alpha_i\right)^1_{i\in\kappa}$, suppose $A\subseteq\omega^\beta\cdot\left(\lambda-1+m\right)+1$ with $\left|A\right|<\lambda$. Then by the argument of case \ref{lopsidedproofquotedcase} we may assume $A=\left\{\omega^\beta\cdot n:n\in S\right\}$ for some $S\subseteq\left\{1,2,\dots,\lambda-1+m\right\}$ with $S\leq\lambda$. Since $\kappa>1$ we have $S\neq\emptyset$, say $s\in S$. Then
\[\left(\bigcup_{n\in\left\{1,2,\dots,\lambda-1+m\right\}\setminus S}\left[\omega^\beta\cdot n+1,\omega^\beta\cdot\left(n+1\right)\right
]\right)\cup\left[\omega^\beta\cdot s+1,\omega^\beta\cdot\left(s+1\right)\right)\]
is a homeomorphic copy of $\omega^\beta\cdot\left(m+1\right)$ disjoint from $A$, which suffices.
\end{enumerate}
\end{enumerate}
\item
This is Corollaries \ref{omega2geq}, \ref{independenceequality}, \ref{independencesupercompact} and \ref{independencemahlo}.
\item
This is Theorem \ref{omega1equal}.
\item
This is Theorem \ref{infiniteofcountable}.
\item
\begin{enumerate}
\item
This is the finite pigeonhole principle.
\item
This follows from Theorems \ref{linkpowers} and \ref{powersroundup} using monotonicity of pigeonhole numbers.
\item
By Lemmas \ref{pseudorelevant} and \ref{pseudopositive}, we may assume that for each $i\in\kappa$, either $\alpha_i=\overline\omega\left[\beta_i,m_i\right]$ or $\alpha_i=\omega^{\beta_i}\cdot\left(m_i+1\right)$ and $\beta_i>0$. It follows that one of Theorems \ref{simplemultiples} and \ref{multiplemixtures} applies, and thus $P^{top}\left(\alpha_i\right)_{i\in\kappa}$ is equal to either $\omega^\beta\cdot m+1$ or $\omega^\beta\cdot\left(m+1\right)$, where $\beta=\beta_0\#\beta_1\#\cdots\#\beta_{\kappa-1}$ and $m=\sum_{i\in\kappa}\left(m_i-1\right)+1$. It remains to determine whether or not $\omega^\beta\cdot m+1\rightarrow\left(top\,\alpha_i\right)^1_{i\in\kappa}$.
\begin{enumerate}
\item
This is the ``only if'' part of Theorem \ref{distinguishing}.
\item\label{proofquotedcase}
\begin{itemize}
\item
If there is no $s\in\kappa$ such that $\alpha_s=\omega^{\beta_s}\cdot\left(m_s+1\right)$, then $\alpha_i=\overline\omega\left[\beta_i,m_i\right]$ for all $i\in\kappa$ and the result is given by Theorem \ref{simplemultiples}.
\item
If there exists $s\in\kappa$ with $\alpha_s=\omega^{\beta_s}\cdot\left(m_s+1\right)$ and $m_i=1$ for all $i\in\kappa\setminus\left\{s\right\}$, then assume without loss of generality that $\operatorname{CB}\left(\beta_s\right)$ is minimal among any $s\in\kappa$ with these properties. By definition of case \ref{proofquotedcase}, there must still exist $t\in\kappa$ such that $\operatorname{CB}\left(\beta_t\right)<\operatorname{CB}\left(\beta_s\right)$, and so the result is given by the ``if'' part of Theorem \ref{distinguishing}.
\item
Otherwise, let $c:\omega^\beta\cdot m+1\rightarrow\kappa$ be a colouring, and assume for simplicity that $\beta_i>0$ for all $i\in\kappa$, the other case being no harder. First note that if $c^{-1}\left(\left\{j\right\}\right)$ contains a homeomorphic copy of $\omega^{\beta_j+1}$ for some $j\in\kappa$, then we are done. Therefore by Lemma \ref{findinspare} we may assume that for each $l\in m$ and each $i\in\kappa$, there exists $Y_{i,l}\subseteq c^{-1}\left(\left\{i\right\}\right)\cap\left[\omega^\beta\cdot l+1,\omega^\beta\cdot\left(l+1\right)\right)$ with $Y_{i,l}\cong\omega^{\beta_i}$. Now by Theorem \ref{simplemultiples}, there exists $j\in\kappa$ and $X\subseteq c^{-1}\left(\left\{j\right\}\right)$ with $X\cong\omega^{\beta_j}\cdot m_j+1$, and moreover by the proof of that theorem we may assume that
\[X\subseteq\bigcup_{l\in S}\left[\omega^\beta\cdot l+1,\omega^\beta\cdot\left(l+1\right)\right]\]
for some $S\subseteq m$ with $\left|S\right|=m_j$.
Two possibilities now remain.
\begin{itemize}
\item
If there exist distinct $s,t\in\kappa$ with $m_s,m_t\geq 2$, then $m>m_i$ for all $i\in\kappa$.
\item
If there exists $s\in\kappa$ with $\alpha_s=\overline\omega\left[\beta_s,m_s\right]$, $m_s\geq2$ and $m_i=1$ for all $i\in\kappa\setminus\left\{s\right\}$, then $m>m_i$ for all $i\in\kappa\setminus\left\{s\right\}$. If $j=s$ then we are done, so we may assume that $j\neq s$.
\end{itemize}
In either case we have $m>m_j$. Therefore there exists $l\in m\setminus S$, whence $X\cup Y_{j,l}$ is a homeomorphic copy of $\omega^{\beta_j}\cdot\left(m_j+1\right)$ in colour $j$, which suffices.\qedhere
\end{itemize}
\end{enumerate}
\end{enumerate}
\end{enumerate}
\end{proof}

\bibliographystyle{plain}
\bibliography{bibliography}

\end{document}